\documentclass[12pt]{amsart}
\pdfoutput=1

\usepackage[english]{babel}
\usepackage[utf8]{inputenc}
\usepackage[top=1.50in,bottom=1.50in, left=1.25in,right=1.25in]{geometry}
\usepackage{amsmath,amssymb}
\usepackage{graphicx}
\usepackage[colorinlistoftodos]{todonotes}
\usepackage{pictexwd,dcpic}
\usepackage{stackengine,scalerel}
\usepackage{mathabx,xypic, stmaryrd,graphicx}
\usepackage{comment}
\usepackage{mathtools}
\usepackage{filecontents}
\usepackage{tikz-cd}
\usepackage{MnSymbol}
\usepackage{fancyhdr}
\usepackage[pdfusetitle,linktocpage,colorlinks=false]{hyperref}
\usepackage[bottom]{footmisc}
\makeatother

\title[The Novikov Conjecture and Group Extensions]{The Novikov conjecture and extensions of coarsely embeddable groups}

\author{Jintao Deng}

\address{Texas A\&M University, Department of Mathematics}
\email{jintaodeng@math.tamu.edu}

\newtheorem{thm}{Theorem}[section]
\newtheorem{prop}[thm]{Proposition}
\newtheorem{lem}[thm]{Lemma}

\newtheorem{defn}[thm]{Definition}
\theoremstyle{remark}
\newtheorem{rem}[thm]{Remark}

\begin{document}
\maketitle

\begin{abstract}
Let $1 \to N \to G \to G/N \to 1$ be a short exact sequence of countable discrete groups and let $B$ be any $G$-$C^*$-algebra. In this paper, we show that the strong Novikov conjecture with coefficients in $B$ holds for such a group $G$ when the normal subgroup $N$ and the quotient group $G/N$ are coarsely embeddable into Hilbert spaces. As a result, the group $G$ satisfies the Novikov conjecture under the same hypothesis on $N$ and $G/N$.
\end{abstract}

\section{Introduction}
The Novikov conjecture is an important problem in higher dimensional topology. It asserts that the higher signatures of a compact smooth manifold are invariant under orientation preserving homotopy equivalences. In the past few decades, noncommutative geometry has provided powerful techniques to study the Novikov conjecture. Using this approach, the Novikov conjecture has been proved for an extensive class of groups (c.f. \cite{Connes}, \cite{CM}, \cite{GHW}, \cite{HK}, \cite{KasSkan}, \cite{Kas_Yu}, \cite{Kas}, \cite{Mis}, \cite{Yu_fin_asym_dim}, \cite{Yu_coa_emb}, \cite{Yu2019}).

The Novikov conjecture is a consequence of the strong Novikov conjecture in the computation of the $K$-theory of group $C^*$-algebras. Given a countable discrete group $G$, there is a universal proper $G$-space $\mathcal{E}G$ which is unique up to equivariant homotopy equivalence (see \cite{BaCoHig}). Let $B$ be any $C^*$-algebra equipped with a $G$-action by $*$-automorphisms. The Baum--Connes assembly map for a countable discrete group $G$ and a $G$-$C^*$-algebra $B$ is a group homomorphism
$$
\mu:KK_*^G(\mathcal{E}G, B) \to K_*(B \rtimes_{\operatorname{r}}G),
$$
where $KK_*^G(\mathcal{E}G, B)$ is the equivariant $K$-homology with $G$-compact supports with coefficients in $B$ of the universal space $\mathcal{E}G$ for proper $G$-actions, and $K_*(B \rtimes_{\operatorname{r}} G)$ is the $K$-theory of the reduced crossed product $B \rtimes_{\operatorname{r}}G$ (see \cite{Kas}). In the special case when $B$ is the complex numbers $\mathbb{C}$ with trivial $G$-action, the Baum--Connes assembly map $\mu$ is a group homomorphism mapping each Dirac type operator to its higher index in $K_*(C^*_{\operatorname{r}}(G))$, where $C_{\operatorname{r}}^*(G)$ is the reduced group $C^*$-algebra. The Baum--Connes conjecture with coefficients in $B$ claims that $\mu$ is an isomorphism, while the strong Novikov conjecture with coefficients in $B$ claims that $\mu$ is injective. When $B$ is the complex numbers $\mathbb C$, this reduces to the usual Baum--Connes conjecture and strong Novikov conjecture, respectively.

Let us recall the concept of coarse embedding which was introduced by Gromov in \cite{Gromov}.
Let $X$ be a metric space, and $\mathcal{H}$ a Hilbert space. A map $\varphi: X \rightarrow \mathcal{H}$ is called a coarse embedding if there exist two non-decreasing functions $\rho_{-},~\rho_{+}: [0.+\infty) \rightarrow [0, +\infty)$, with $\rho_{-}(t) \leq \rho_{+}(t)$ for all $t\geq 0$, and $\displaystyle \lim_{t \to +\infty} \rho_{-}(t)=+\infty$, such that
\begin{center}
$\rho_{-}(d(x,y)) \leq \|\varphi(x)-\varphi(y)\|\leq \rho_{+}(d(x,y))$ for all $x,y \in X$.
\end{center}

Yu \cite{Yu_coa_emb} and Skandalis--Tu--Yu \cite{Skan_Tu_Yu} proved the Novikov conjecture for any group which admits a coarse embedding into Hilbert space. In \cite{Kas_Yu}, Kasparov and Yu strengthened this result, showing that the Novikov conjecture holds for groups which are coarsely embeddable into Banach spaces with property $(H)$.

In \cite{MR1817505}, Oyono-Oyono established a group extension result for the Baum--Connes conjecture. Let $N$ and $G$ be countable discrete groups, and $N$ a normal subgroup of $G$. Oyono-Oyono showed that if the quotient group $G/N$ and all subgroups of $G$ containing  $N$ with finite index satisfy the Baum--Connes conjecture, then $G$ satisfies the Baum--Connes conjecture. With this extension result, one can show that the Baum--Connes conjecture holds for a large class of groups. For instance, based on Higson and Kasparov's result on the Baum--Connes conjecture for a-T-menable groups (\cite{HK}), the result of Oyono-Oyono implies that the Baum--Connes conjecture holds for all extensions of a-T-menable groups.

To obtain an extension result for the Novikov conjecture, one might attempt to show that coarse embeddability into Hilbert space is closed under taking group extensions, and apply Yu's result (\cite{Yu_coa_emb}).  However, in \cite{Arth_Tessera}, Arzhantseva and Tessera constructed a finitely generated group $G$ which is not coarsely embeddable into Hilbert space, but has a normal subgroup $N$ such that $N$ and $G/N$ are coarsely embeddable into Hilbert spaces. Note that every subgroup of $G$ containing $N$ with finite index is also coarsely embeddable into Hilbert space. To obtain an analogue of the extension result of the Baum--Connes conjecture (\cite{MR1817505}), other techniques are needed to show that the group obtained from extension of coarsely embeddable groups satisfies the Novikov conjecture.

Our main results are the following.
\begin{thm}\label{main_result}
Let $1 \to N \to G \to G/N\to 1$ be a short exact sequence of countable discrete groups and $B$ a $G$-$C^*$-algebra. If $N$ and $G/N$ are coarsely embeddable into Hilbert spaces, then the strong Novikov conjecture holds for $G$ with coefficients in the $G$-$C^*$-algebra $B$, that is, the Baum--Connes assembly map
$$
\mu: KK^G_*(\mathcal{E}G, B) \to K_*(B \rtimes_{\operatorname{r}} G)
$$is injective, where $\mathcal{E}G$ is the universal space for proper $G$-action, and $B \rtimes_{\operatorname{r}} G$ is the reduced crossed product $C^*$-algebra.
\end{thm}

Let $G$ be a countable discrete group with a coarsely embeddable normal subgroup $N\leq G$. Assume that the quotient group $G/N$ is also coarsely embeddable into Hilbert space. It follows from Theorem \ref{main_result} that the rational strong Novikov conjecture holds for $G$, that is, the Baum--Connes assembly map
$$
\mu: KK^G_*(\mathcal{E}G)\otimes\mathbb{Q} \to K_*(C^*_{\operatorname{r}} G) \otimes\mathbb{Q}
$$
is injective. We remark that the rational strong Novikov conjecture implies the Novikov conjecture on the homotopy invariance of higher signatures and the Gromov-Lawson-Rosenberg conjecture regarding the existence of positive scalar curvature on closed aspherical manifolds.

This paper is organized as follows. In Section 2, we introduce Roe algebras, the local index map and its relation to the Baum--Connes map. In Section 3, we recall the $C^*$-algebra associated with an infinite dimensional Euclidean space and the generalization to the field case. In Section 4, we define an a-T-menable groupoid associated to a coarsely embeddable group, and then define certain twisted Roe algebras and twisted localization algebras. In Section 5, we prove that the evaluation-at-zero map induces an isomorphism from the $K$-theory of twisted localization algebras to the $K$-theory of twisted Roe algebras for groups which are extensions of coarsely embeddable groups. In Section 6, we construct a geometric analogue of Higson--Kasparov--Trout's Bott map from the $K$-theory of localization algebras to the $K$-theory of twisted localization algebras. This geometric analogue of the Bott map is then used to reduce the Novikov conjecture to the twisted Baum--Connes conjecture for groups which are extensions of coarsely embeddable groups.

\section{The Baum--Connes map and localization}

In this section, we will first recall the definition of Roe algebras, and the Baum--Connes assembly map. We then move on to define the local index map and show the connection between the local index map and the Baum--Connes assembly map.

\subsection{Roe algebras}
Let $G$ be a countable discrete group, and $\Delta$ a locally compact metric space with a proper cocompact $G$-action. The action is proper if the map
\begin{center}
 $\Delta \times G \to \Delta \times \Delta$, via $(x, g) \mapsto (x, gx)$,
\end{center}
  is a proper map. A $G$-action is said to be cocompact if there exists a compact subset $\Delta_0 \subset \Delta$ such that $G \cdot \Delta_0=\Delta$. Let $C_0(\Delta)$ be the $C^*$-algebra of all continuous functions on $\Delta$, which vanish at infinity. Let $B$ be any $G$-$C^*$-algebra.

\begin{defn}[\cite{WilYu}]
Let $H$ be a Hilbert module over $B$ and let $\pi: C_0(\Delta)\to B(H)$ be a $*$-homomorphism from $C_0(\Delta)$ to $B(H)$, where $B(H)$ is the algebra of all adjointable operators on $H$. Let $T \in B(H)$ be an adjointable operator on $H$.
\begin{enumerate}
\item The support of $T$, denoted by $supp(T)$, is defined to be the complement (in $\Delta \times \Delta$) of all pairs $(x,y) \in \Delta \times \Delta$ for which there exist $f, g \in C_0(\Delta)$ with $f(x)\neq 0$ and $g(y) \neq 0$ such that $\pi(f)T \pi(g)=0$.
\item The propagation of $T$ is defined to be
$$\mbox{propagation}(T)=\sup\{d(x,y):(x,y) \in \text{supp}(T)\}.$$
If $\text{propagation}(T) <\infty$, the operator $T$ is said to have finite propagation.
\item The operator $T$ is said to be locally compact if $\pi(f)T$ and $T\pi(f)$ are in $K(H)$ for all $f \in C_0(\Delta)$, where $K(H)$ is defined to be the operator norm closure of all finite-rank operators on the Hilbert module $H$.
\end{enumerate}
\end{defn}

Let $H$ be a countably generated Hilbert module over $B$, and $U: G \to \mathcal{U}(H)$ a unitary representation of $G$. A $*$-homomorphism $\pi: C_0(\Delta) \to B(H)$ is said to be covariant if $\pi(\gamma f)=U_{\gamma} \pi(f) U_{\gamma^{-1}}$, for all $\gamma \in G$, $f \in C_0(\Delta)$. The triple $(C_0(\Delta), G, \pi)$ is called a covariant system. An operator $T \in B(H)$ is said to be $G$-invariant if $U_{\gamma}TU_{\gamma^{-1}}=T$, for all $\gamma \in G$. Let us also recall the definition of an admissible covariant system, more details can be found in \cite{WillettYuBook}.
\begin{defn}[\cite{WillettYuBook}]\label{DefnAdmissible}

A covariant system $(C_0(\Delta), G, \pi)$ is said to be admissible if
\begin{enumerate}
\item $H$ is isomorphic to $H_{\Delta}\otimes E \otimes B$ as $G$-Hilbert modules over $B$, where $H_{\Delta}$ and $E$ are Hilbert spaces,
\item  $\pi=\pi_0\otimes 1 $ for some $G$-equivariant $*$-homomorphism $\pi_0: C_0(\Delta) \to B(H_{\Delta})$, such that $\pi_0(f)$ is not in $K(H_{\Delta})$ for any non-zero function $f \in C_0(\Delta)$, and $\pi_0$ is non-degenerate in the sense that $\{\pi_0 \big(C_0(\Delta)\big)H_{\Delta}\}$ is dense in $H_{\Delta}$,
\item for every finite subgroup $F$ of $G$ and every $F$-invariant Borel subset $U$ of $\Delta$, $E$ is isomorphic to $\ell^2(F) \otimes H_U$ as $F$-Hilbert spaces for some Hilbert space $H_U$ with a trivial $F$-action.
\end{enumerate}
\end{defn}

Let $G$ be a countable discrete group, and $\Delta$ a locally compact metric space with a proper cocompact $G$-action and let $B$ be a $G$-$C^*$-algebra. There is always an admissible covariant system. Chose an infinite-dimensional separable Hilbert space $H_0$ and a countable dense $G$-invariant subset $X \subset \Delta$, then define
 $$H=\ell^2(X) \otimes H_0 \otimes \ell^2(G) \otimes B.$$
The tensor product $H$ is a Hilbert $B$-module with the $B$-valued inner product
$$\left\langle u \otimes a, v \otimes b\right\rangle =\left\langle u, v \right\rangle \cdot a^*b,$$
for all $u, v \in \ell^2(X)\otimes H_0 \otimes \ell^2(G)$ and $a, b \in B$.
The space $H$ is equipped with a right $B$-action by
$$\left(u \otimes a\right)b=u \otimes ab,$$
for all $u \otimes \ell^2(X) \otimes H_0\otimes \ell^2(G)$ and $a, b \in B$.
 In addition, the Hilbert $B$-module is endowed with the diagonal action of $G$ by
 $$U_g: \delta_z \otimes v \otimes \delta_h \otimes a \mapsto \delta_{gz} \otimes v \otimes \delta_{gh} \otimes g \cdot a,$$
 where $g, h \in G$, $a \in B$, $z \in X$. Define an action of $C_0(\Delta)$ by pointwise multiplication
 $$f: \delta_z \otimes v \otimes \delta_h \otimes a \mapsto f(z)\delta_z \otimes v \otimes \delta_h \otimes a.$$
Following Lemma 4.5.5 in \cite{WillettYuBook}, it is obvious that $\left(C_0(\Delta), G\right)$ is an admissible system.

Now we are ready to define the Roe algebra, following Roe \cite{Roe93}.
\begin{defn}
Let $\big(C_0(\Delta), G, \pi\big)$ be an admissible covariant system. The algebraic Roe algebra with coefficients in $B$, denoted by $C^*_{alg}(\Delta, G, B)$, is defined to the algebra of all the $G$-invariant, locally compact operators in $B(H)$ with finite propagation. The Roe algebra with coefficients in $B$, denoted by $C^*(\Delta, G, B)$, is the norm closure of $C^*_{alg}(\Delta, G, B)$ under the operator norm on $H$.
\end{defn}

Next, we will recall some basic properties of Roe algebras.
Let $\Delta_1$, $\Delta_2$ be two locally compact metric spaces with proper and isometric $G$-actions. A Borel map $f: \Delta_1 \to \Delta_2$ is $G$-equivariant if $f(g x)=g f(x)$, for all $x \in \Delta_1$, $g \in G$. The map $f: \Delta_1 \to \Delta_2$ is said to be a coarse embedding if there exist non-decreasing functions $\rho_-, \rho_+: [0, \infty) \to [0, \infty)$ such that

\begin{enumerate}
\item $\displaystyle\lim_{t \to +\infty}\rho_-(t)= +\infty$,

\item $\rho_-(d(x, y)) \leq d(f(x), f(y)) \leq \rho_+(d (x, y))$, for all $x, y \in \Delta_1$.
\end{enumerate}

Given an equivariant coarse embedding, we will define an isometry between admissible Hilbert modules covering the map. Let
 $(C_0(\Delta_1), \pi_1, G)$ and $(C_0(\Delta_2), \pi_2, G)$ be admissible systems on $H_1=H_{\Delta_1} \otimes E_1 \otimes B $ and $H_2=H_{\Delta_2} \otimes E_2 \otimes B $ as in Definition \ref{DefnAdmissible}. For any adjointable operator $T: H_1 \to H_2$, the support, denoted by $\text{supp}(T)$, is defined to be the complement (in $\Delta_2 \times \Delta_1$) of all the pairs $(x,y) \in \Delta_2 \times \Delta_1$ for which there exist $f \in C_0(\Delta_1)$ and $g \in C_0(\Delta_2)$ with $f(y)\neq 0$ and $g(x) \neq 0$ such that
$\pi_2(g)T \pi_1(f)=0.$

A space $\Delta_1$ is said to be $G$-equivariantly coarsely equivalent to $\Delta_2$, if there exist $G$-equivariant coarse embedding $f: \Delta_1 \to \Delta_2$ and $g: \Delta_2 \to \Delta_1$, such that $d(fg(y), y)<c$ for all $y \in \Delta_2$ and $d(gf(x), x)<c$ for all $x \in \Delta_1$, where $c$ is a positive constant. Let us recall the result that the $K$-theory of Roe algebras with coefficients in any $G$-$C^*$-algebra $B$ is invariant under equivariant coarse equivalence. For completeness, we also present the proof.

\begin{prop}[\cite{WillettYuBook}]\label{Prop_Roe_coarse_invar}
Let $\Delta_1$ and $\Delta_2$ be metric spaces with proper $G$-actions. If $\Delta_1$ is $G$-equivariantly coarsely equivalent to $\Delta_2$, then $K_*(C^*(\Delta_1, G, B))$ is isomorphic to $K_*(C^*(\Delta_2, G, B))$, for any $G$-$C^*$-algebra $B$.
\end{prop}

 \begin{proof}
 Since the $G$-action  on $\Delta_2$ is proper and isometric, by Lemma A.2.8 in \cite{WilYu},
  one can find a Borel cover $\{U_i\}$ with mutually disjoint elements, such that
\begin{enumerate}
\item $U_i=K_i \times_{G_{i}} G$, where  $G_i\leq G$ is a finite subgroup, and $K_i\subset \Delta_1$ is $G_i$-invariant for all $i$,
\item $K_{i}$ has non-empty interior for all $i$,
\item the diameter of $K_i$ is uniformly bounded for all $i$.
\end{enumerate}
One obtains a cover $\{f^{-1}(U_i)\}$ of $\Delta_1$. The representation of $C_0(\Delta_1)$ on $H_1$ extends to a representation of the algebra of all bounded Borel functions on $\Delta_1$. Thus, for each $i$, we have $\chi_{U_i}H_2\cong \chi_{U_i}H_{\Delta_2} \otimes \ell^2(G_i) \otimes H_{U_i} \otimes B$ by Definition \ref{DefnAdmissible}. Since $K_i$ has non-empty interior and is $G_i$-invariant, we can define a $G_i$-equivariant isometry $V_i: \chi_{f^{-1}(K_i)}H_1 \to \chi_{K_i}H_2$. Hence, we obtain a $G$-equivariant isometry $V_i: \chi_{f^{-1}(U_i)}H_1 \to \chi_{U_i}H_2 $, which, in return, gives us an isometry $V=\bigoplus_{i} V_i:H_1 \to H_2$. It follows from condition (3) above that the operator $VTV^*$ has finite propagation when $T$ has finite propagation. Therefore, the map
$$
\mbox{Ad}(V): C^*(\Delta_1, G, B) \to C^*(\Delta_2, G, B)
$$
given by
$\mbox{Ad}(V)(T)=VTV^*$ is well-defined and induces a homomorphism on $K$-theory
$$\text{Ad}(V)_*:  K_*(C^*(\Delta_1, G, B)) \to K_*(C^*(\Delta_2, G, B)).$$
Similarly, the equivariant coarse map from $\Delta_2$ to $\Delta_1$ gives rise to an inverse map.

\end{proof}

\begin{rem}\label{rem_Roe_independent_on_admissible_system}
It is easy to check that when $f$ is the identity map, the isometry above is a unitary, thus, there is an isomorphism between Roe algebras defined on different admissible covariant systems. As a result, the definition of Roe algebras is independent of the choice of the admissible covariant system.
\end{rem}

The following result is essentially due to John Roe.
\begin{prop}[\cite{Kas_Yu}]
Let $G$ be a countable discrete group, and $\Delta$ a locally compact metric space with a proper cocompact $G$-action. If $\big(C_0(\Delta), G, \pi \big)$ is an admissible covariant system, then the Roe algebra $C^*(\Delta, G, B)$ is $*$-isomorphic to $(B \rtimes _{\operatorname{r}} G)\otimes \mathcal{K}$, where $\mathcal{K}$ is the algebra of all compact operators on some infinite-dimensional separable Hilbert space.
\end{prop}

\subsection{The Baum--Connes assembly map}
Let $H$ be a $G$-Hilbert module over $B$. Let $F$ be an operator in $B(H)$, and let $\pi: C_0(\Delta) \to B(H)$ be a $*$-representation of $C_0(\Delta)$, such that $F$ is $G$-invariant, $\pi(f)F -F\pi(f)$, $\pi(f)(FF^*-1)$ and $(F^*F-1)\pi(f)$ are in $K(H)$ for all $f \in C_0(\Delta)$. The triple $(H, \pi, G)$ is called a $KK$-cycle.

The group $KK_0^G(\Delta, B)$ is an abelian group consisting of the homotopy equivalence classes of $KK$-cycles.
By Proposition 5.5 in \cite{KasSkan}, any class in $KK_0^G(\Delta, B)$ can be represented by a $KK$-cycle $(H, \pi, F)$ such that the covariant system $(C_0(\Delta), G, \pi)$ is admissible, where $H$ is $G$-Hilbert module over $B$, $F$ is an operator in $B(H)$, such that $F$ is $G$-invariant and $\pi(f)F-F\pi(f)$, $\pi(f)(FF^*-1)$, and $\pi(f)(F^*F-1)$ are in $K(H)$ for all $f \in C_0(\Delta)$.

For any fixed $\epsilon>0$, let $\{U_i\}_{i \in I}$ be a locally finite and $G$-equivariant open cover of $\Delta$ such that $\text{diameter}(U_i) < \epsilon$ for all $i \in I$. An open cover is said to be $G$-equivariant if $g(U_i)\in \{U_i\}_{i \in I}$, for all $g \in G$, $i \in I$. Let $\left\{\phi_i\right\}_{i \in I}$ be a $G$-equivariant partition of unity subordinate to the open cover $\{U_i\}_{i \in I}$. A partition of unity $\left\{\phi_i\right\}_{i \in I}$ is said to be $G$-equivariant if $g\cdot \phi_i \in \left\{\phi_i\right\}_{i \in I}$, for all $g \in G$, $i \in I$.

Define an operator on $H$ by
$$
F_{\epsilon}:=\sum_{i \in I} \pi(\sqrt{\phi_i})F\pi(\sqrt{\phi_i}),
$$
where the sum converges in the strong operator topology.

Note that the propagation of $F_{\epsilon}$ is smaller than $\epsilon$, and $(H, \pi, F_{\epsilon})$ is equivalent to $(H, \pi, F)$ in $KK_0^G(\Delta, B)$, for any $\epsilon>0$. By the definition of $F_{\epsilon}$, $F_{\epsilon}$ is a multiplier of $C^*(\Delta, G, B)$, and it is invertible modulo $C^*(\Delta, G, B)$. Let $M(C^*(\Delta, G, B))$ be the multiplier algebra of $C^*(\Delta, G, B)$. Then we have the boundary map in K-theory
$$
\partial:K_1(M((C^*(\Delta, G, B))/C^*(\Delta, G, B)) \to K_0(C^*(\Delta, G, B)).
$$
We define the Baum--Connes assembly map for $G$
$$
\mu: KK_0^G(\Delta, B) \to K_0(C^*(\Delta, G, B)) \cong K_0(B \rtimes_{\operatorname{r}}G)
$$
by
$$
\mu([(H, \pi, F)])=\partial([F_{\epsilon}]).
$$
It is not difficult to check that the map $\mu$ is well-defined.

Similarly, we can define the Baum--Connes assembly map
$$
\mu: KK_1^G(\Delta, B) \to K_1(C^*(\Delta, G, B)) \cong K_1(B \rtimes_{\operatorname{r}}G).
$$
This induces the Baum--Connes assembly map
$$
\mu: KK_*^G(\mathcal{E}G, B) \to  K_*(B \rtimes_{\operatorname{r}}G),
$$
where $KK_*^G(\mathcal{E}G, B)$ is defined to be the inductive limit of $KK_*^G(\Delta, B)$ over all $G$-invariant and cocompact subspaces $\Delta$ of $\mathcal{E}G$. Later, we will show that these invariant and cocompact subspaces can be chosen to be finite-dimensional simplicial complexes, since there exist simplicial models for the universal space $\mathcal{E}G$.

By Proposition 1.8 in \cite{BaCoHig}, one can choose a model for the universal space $\mathcal{E}G$ for proper $G$-action as follows.

\begin{defn}
For each $d>0$, we define the Rips complex, denoted by $P_d(G)$, to be the simplicial complex with vertex set $G$ and such that a finite subset $\{\gamma_i\}_{i=1}^n \subset G$ spans a simplex if and only if $d(\gamma_i, \gamma_j)\leq d$.
\end{defn}

For each $d>0$, the Rips complex $P_d(G)$ is endowed with a spherical metric as follows. For the simplex $\left\{\sum_{i=0}^nc_i \gamma_i: c_i \in [0,1],~\text{with}~\sum_{i=1}^n c_i=1\right\}$ spanned by the finite subset $\{\gamma_i\}_{i=1}^n \subset G$, it is identified with the part of the $(n-1)$-sphere in the positive orthant via the map
$$
\sum_{i=0}^nc_i \gamma_i \mapsto \left(c_0/\left(\sum_{i=0}^nc_i^2\right) , c_1/\left(\sum_{i=0}^nc_i^2\right), \cdots, c_n/\left(\sum_{i=0}^nc_i^2\right)\right).
$$
The simplex spanned by $\{\gamma_i\}_{i=1}^n \subset G$ admits a metric induced by pullback of the standard Riemannian metric on the sphere. The spherical metric on the Rips complex $P_d(G)$ is the maximal metric such that it restricts to the above spherical metric on each simplex.

We can choose the union $\bigcup_{d>0} P_d(G)$ as a model of $\mathcal{E}G$, where the union is equipped with the weak topology under which a subset $C \subset \bigcup_{d>0} P_d(G)$ is closed if and only if $C\cap P_d(G)$ is closed for each $d>0$.

\subsection{Localization algebras and the local index map}
Let us now recall localization algebras and the local index map, and introduce some basic properties of the $K$-theory of localization algebras.

Let $\Delta$ be the topological realization of a locally compact and finite-dimensional simplicial complex endowed with the simplicial metric. Let $(C_0(\Delta), G, \pi)$ be an admissible covariant system, where $\pi: C_0(\Delta) \to B(H)$ is a $*$-homomorphism for some Hilbert module $H$ over $B$.

\begin{defn}\leavevmode
\begin{enumerate}
\item The algebraic localization algebra $C^*_{L, alg}(\Delta, G, B)$ is defined to be the algebra of all the bounded and uniformly continuous maps $f: [0, \infty) \to C^*_{alg}(\Delta, G, B)$ such that $\text{propagation}(f(t)) \to 0$ as $t \to \infty$.

\item The localization algebra $C^*_L(\Delta, G, B)$ is the norm closure of $C^*_{L, alg}(\Delta, G, B)$ under the norm
$$
\|f\|=\sup_{t\in [0, \infty)}\|f(t)\|.
$$
\end{enumerate}
\end{defn}
The localization algebra is an equivariant analogue of the algebra introduced by Yu in \cite{Yu_localization}.
Note that, up to $*$-isomorphism, the localization algebra $C_L^*(\Delta, G, B)$ is independent of the choice of the admissible covariant system by Remark \ref{rem_Roe_independent_on_admissible_system}.

A Borel coarse map $f: \Delta_1 \to \Delta_2$ is said to be Lipschitz, if there exists a constant $c>0$, such that $d(f(x), f(y)) \leq c d(x, y)$ for all $x, y \in \Delta_1$. A Lipschitz coarse map $f:\Delta_1 \to \Delta_2 $ induces a homomorphism $\text{Ad}({V_f}): C_L^*(\Delta_1, G, B) \to C_L^*(\Delta_2, G, B)$ as follows.

Let $\{\epsilon_n\}_{n \in \mathbb{N}}$ be a sequence of positive numbers with $\lim_{n \to \infty} \epsilon_n=0$. By the same argument as Proposition \ref{Prop_Roe_coarse_invar}, for each $k$, there exists a $G$-equivariant isometry $V_k: H_1 \to H_2$ between the Hilbert $B$-module, such that
$$
\text{supp}(V_k) \subset \left\{(y, x) \in \Delta_2 \times \Delta_1: d(y, f(x)) \leq \epsilon_k\right\}.
$$

Define a family of isometries $\big(V_f(t)\big)_{t \in [0, \infty)}$ from $H_1$ to $H_2$ by
$$
V_f(t):=R(t-k)(V_k \oplus V_{k+1}) R^*(t-k),
$$
for $t \in [k, k+1)$, where
$$
R(t)=
\left(
\begin{matrix}
\cos(\pi t /2)& \sin(\pi t /2)\\
-\sin(\pi t/2)& \cos(\pi t /2)
\end{matrix}
\right).
$$
Then $V_f(t)$ induces a homomorphism on unitization
$$\text{Ad}(V_f(t)):\big(C_L^*(\Delta_1, G, B)\big)^+ \to \big(C_L^*(\Delta_2, G, B)\big)^+ \otimes M_2(\mathbb{C})$$
by
$$
\text{Ad}(V_f(t))(u(t)+cI)=V_f(t)(u(t) \oplus 0)V_f^*(t)+cI
$$
for all $u \in C_L^*(\Delta_1, G, B)$ and $c \in \mathbb{C}$.

The family of maps  $\mbox{Ad}(V_f(t))$ induces a homomorphism,
$$\mbox{Ad}(V_f(t))_*: K_*(C_L^*(\Delta_1, G, B))\to K_*(C_L^*(\Delta_2, G, B))$$
on the $K$-theory. Note that $\mbox{Ad}(V_f(t))(u(t)+cI)$ is uniformly continuous on $t$, even though $V_f(t)$ is not continuous. It is also easy to check the propagation condition of the path $\mbox{Ad}(V_f(t))(u(t)+cI)$.
By Lemma 3.4 of \cite{Yu_localization}, the homomorphism between the $K$-groups is independent of the choice of the family of isometries $\{ V_k \}_k$.

\begin{defn}
Let $\Delta_1$ and $\Delta_2$ be two proper metric spaces and $f$, $g$ two Lipschitz maps from $\Delta_1$ to $\Delta_2$. The map $f$ is said to be strongly Lipschitz homotopy equivalent to $g$ if there exists a continuous homotopy $F(t, x): \left[0, 1\right]\times \Delta_1 \to \Delta_2$, such that
\begin{enumerate}
\item $F(t, x)$ is a coarse map from $\Delta_1$ to $\Delta_2$ for each $t$,

\item there exists a constant $C>0$, such that $d(F(t, x), F(t, y))\leq Cd(x, y)$ for all $x, y \in \Delta_1$ and $t \in [0, 1]$,

\item for any $\epsilon>0$, there exists $\delta>0$, such that $d(F(t_1, x), F(t_2, x)) \leq \epsilon$ for all $x \in X$, and $|t_1-t_2|< \delta$,

\item $F(0, x)=f(x)$, $F(1, x)=g(x)$, for all $x \in X$.
\end{enumerate}
\end{defn}

\begin{defn}
The metric  space $\Delta_1$ is said to be strongly Lipschitz homotopy equivalent to $\Delta_2$ if there exist two Lipschitz coarse maps $f: \Delta_1 \to \Delta_2$ and $g: \Delta_2 \to \Delta_1$ such that $f\circ g$ and $g\circ f$ are respectively strongly Lipschitz homotopy equivalent to $id_{\Delta_2}$ and $id_{\Delta_1}$.
\end{defn}

The $K$-theory of localization algebras is invariant under strong Lipschitz homotopy equivalence. Following the proof of Proposition 3.7 in \cite{Yu_localization}, it is not difficult to prove the following equivariant analogue. For ease of reference later, we include the proof.
\begin{prop}[\cite{Yu_localization}]\label{thm_Lip_hom_inva}
Let $f: \Delta_1 \to \Delta_2$ be a strongly Lipschitz homotopy equivalence, then the map
$$\mbox{Ad}(V_f(t))_*: K_*(C_L^*(\Delta_1, G, B)) \to K_*(C_L^*(\Delta_2, G, B))$$
is an isomorphism.
\end{prop}

\begin{rem}\label{rem_localization_infinite_plus_mod}
Let $(C_0(\Delta_1), \pi, G)$ be an admissible covariant system on a Hilbert module $H_1=H_{\Delta_1} \otimes E \otimes B$. Set $H'_1=\oplus_{k=1}^{\infty} H_{\Delta_1} \otimes E \otimes B$. Define a natural homomorphism $\eta: C_L^*(\Delta_1, H_1, G, B) \to C_L^*(\Delta_2, H'_1, G, B)$ via $\eta(b)=b \oplus 0$. It is easy to show that $\eta$ induces an isomorphism on $K$-theory of localization algebras defined on different admissible covariant systems.
\end{rem}

\begin{proof}
It suffices to show that the homomorphism $\text{Ad}(V_{gf}(t))_*$ is the identity map on $K$-theory. Let $F(x, t): \Delta_1 \times [0, 1] \to \Delta_2$ be the strong Lipschitz homotopy equivalence with $F(x, 0)=(gf)(x)$, $F(x, 1)=x$, for all $x \in X$. Fix a sequence of positive numbers $\{\epsilon_n\}_{n \in \mathbb{N}}$ with $\lim_{n \to \infty} \epsilon_n=0$ and a sequence of non-negative numbers $\{t_{i,j}\}_{i,j=0}^{\infty}$, satisfying

\begin{itemize}
\item $t_{0, j}=0$, $t_{i+1, j}\geq t_{i, j}$, for all $i, j \geq 0$,
\item for each $j$, $\exists$ $N_j$, such that $t_{i, j}=1$ for all $i\geq N_j$,
\item $d(F(x, t_{i, j}), F(x, t_{i+1, j}))\leq \epsilon_j$ , and $d(F(x, t_{i, j}), F(x, t_{i, j+1}))\leq \epsilon_j$, for all $x \in X$.
\end{itemize}
We shall prove that $\text{Ad}_*(V_{F(\cdot, t_{0, j})}(t))=\text{Id}$ for $K_1$-case. The $K_0$-case can be dealt with in a similar way by a suspension argument. Notice that $C_L(\Delta_1, G, B) \otimes M_n(\mathbb{C})\cong C_L^*(\Delta_1, G, B)$ for all $n$. Thus, every element in $K_*(C_L^*(\Delta_1, G, B))$ can be represented by an invertible element $u$ in $\big(C_L^*(\Delta_1, G, B)\big)^+$. Let $v=u \oplus I \in \big(C_L^*(\Delta_1, G, B)\big)^+\otimes M_2(\mathbb{C})$. Consider the following invertible elements in $\big(C_L^*(\Delta_1, G, B)\big)^+\otimes M_2(\mathbb{C})$.
\begin{align*}
a &=\oplus_{k\geq 0}Ad_*(V_{F(\cdot, t_{k, i})(t)})(u)v^{-1},\\
b &=\oplus_{k\geq 0}Ad_*(V_{F(\cdot, t_{k+1, i})(t)})(u)v^{-1},\\
c &=\oplus_{k\geq 1}Ad_*(V_{F(\cdot, t_{k, i})(t)})(u)v^{-1}.
\end{align*}
By the definitions of operators $\text{Ad}_{F(\cdot, t_{i, j})}$ and the sequence $\left\{t_{i, j}\right\}$, it is not difficult to check that $a$, $b$, $c$ are indeed elements in $\big(C^*_L(\Delta_2, G, B)\otimes M_2(\mathbb{C})\big)^+$. Obviously, we have that $a$ is equivalent to $b$, and $b$ is equivalent to $c$ in $K_1(C_L(\Delta_2, G, B))$. It follows that $\text{Ad}(V_{F(\cdot, t_{0, i})}(t))(u)v^{-1}=a\oplus_{k \geq 1} I=bc^{-1}$ which is equivalent to $\oplus_{k \geq 0}I$ in $K_1(C_L^*(\Delta_2, G, B))$. By Remark \ref{rem_localization_infinite_plus_mod}, we have that $\text{Ad}(V_{F(\cdot, t_{0, i})}(t))(u)$ is equivalent to $v$ in $K_1(C_L(\Delta_2, G, B))$.
\end{proof}


The following Mayer--Vietoris sequence is an equivariant analogue of the Mayer--Vietoris sequence introduced by Yu, and more details can be found in \cite{HRY}.

\begin{prop}[\cite{HRY}]\label{prop_MV squence_Local}
Let $\Delta$ be a simplicial complex endowed with the spherical metric, and let $G$ be a countable discrete group. Assume $G$ acts on $\Delta$ properly by isometries. Let $X_1, X_2\subset \Delta$ be $G$-invariant simplicial subcomplexes endowed with subspace metric. Then we have the following six-term exact sequence:
\begin{equation*}
\begin{tikzcd}
	K_0(L_{X_1\cap X_2, B})\arrow[r] & K_0(L_{X_1, B}) \oplus K_0(L_{X_2, B})  \arrow[r] & K_0(L_{X_1\cup X_2, B})\arrow[d]\\
	K_1(L_{X_1 \cup X_2, B})\arrow[u]\arrow[r] & K_1(L_{X_1, B}) \oplus K_1(L_{X_2, B})\arrow[r]& K_1(L_{X_1\cap X_2, B}),
\end{tikzcd}
\end{equation*}
where we set $L_{X_1, B}=C^{\ast}_L(X_1, G, B)$, $L_{X_2, B}=C^{\ast}_L(X_2, G, B)$,
$L_{X_1 \cup X_2, B}=C^{\ast}_L(X_1\cup X_2, G, B)$, and $L_{X_1 \cap X_2, B}=C^{\ast}_L(X_1\cap X_2, G, B)$ for brevity.
\end{prop}
\begin{rem}
It is easy to verify that the above exact sequence is natural with regard to the coefficient algebra $B$ in the following sense. If $\varphi: B \to B'$ is a $G$-equivariant $*$-homomorphism between $G$-$C^*$-algebras $B$ and $B'$, then it induces $*$-homomorphism between the Roe algebras $\varphi: C^*(Y, G, B) \to C^*(Y, G, B')$ by
$$\varphi\left((T_{y, z})_{y, z \in Y}\right)=(\varphi(T_{y, z}))_{y, z \in Y},$$
for each $(T_{y, z})_{y, z \in Y} \in C^*(Y, G, B)$, where $Y \subset \Delta$ be any countable $G$-invariant set. Obviously, this map induces a homomorphism on the localization algebras $\varphi_L: C_L^*(Y, G, B) \to C_L^*(Y, G, B')$. As a result, we obtain a map
$$\varphi_{L, *}: K_*(C_L^*(Y, G, B)) \to K_*(C_L^*(Y, G, B'))$$
 induced by $\varphi_L$ on $K$-theory.

The exact sequence in Proposition \ref{prop_MV squence_Local} is natural with respect to coefficient algebras in the sense that the diagram
	\begin{equation*}
\xymatrix{
	\cdots\ar[r]&K_0(L_{X_1\cap X_2, B}) \ar[r]\ar[d]^{\varphi_{L, *}}& K_0(L_{X_1, B}) \oplus K_0(L_{X_2, B})\ar[r]\ar[d]^{\varphi_{L, *}}&K_0(L_{X_1\cup X_2, B})\arrow[d]\ar[d]^{\varphi_{L,*}}\ar[r]&\cdots\\
	\cdots\ar[r]&K_0(L_{X_1\cap X_2, B'})\ar[r]&K_0(L_{X_1, B'}) \oplus K_0(L_{X_2, B'})\ar[r]&K_0(L_{X_1\cup X_2, B.})\ar[r]&\cdots.
	}
	\end{equation*}
commutes, where we set $L_{X_1, B}=C^{\ast}_L(X_1, G, B)$, $L_{X_2, B}=C^{\ast}_L(X_2, G, B)$,
$L_{X_1 \cup X_2, B}=C^{\ast}_L(X_1\cup X_2, G, B)$, $L_{X_1 \cap X_2, B}=C^{\ast}_L(X_1\cap X_2, G, B)$
$L_{X_1, B'}=C^{\ast}_L(X_1, G, B')$, $L_{X_2, B'}=C^{\ast}_L(X_2, G, B')$,
$L_{X_1 \cup X_2, B'}=C^{\ast}_L(X_1\cup X_2, G, B')$, and $L_{X_1 \cap X_2, B'}=C^{\ast}_L(X_1\cap X_2, G, B')$
 for brevity.

\end{rem}
Let us now define the local index map. For every positive integer $n$, let $\{U_{n, i}\}_{i\in I}$ be a locally finite and $G$-equivariant open cover for $\Delta$ with $\text{diameter}(U_{n, i}) \leq \frac{1}{n}$ for all $i$. Let $\left\{\phi_{n, i}\right\}_{i\in I}$ be the partition of unity subordinate to the open cover $\{U_{n, i}\}_i$. Let $[\Delta, \pi, F]\in KK_0^G(\Delta, B)$. Define an operator-valued function $F(t)$ on $[0, \infty)$ by
$$
F(t)=\sum_{i}(1-(t-n)) \pi(\sqrt{\phi_{n, i}}) F\pi(\sqrt{\phi_{n, i}})+(t-n) \pi(\sqrt{\phi_{n+1, i}}) F\pi(\sqrt{\phi_{n+1, i}})
$$
for all $t \in [n, n+1)$, where the sum converges in the strong operator topology. Note that
$$\text{propagation}(F(t))\to 0$$
as $t \to \infty$. We obtain a multiplier $(F(t))_{t \in [0, \infty)}$ of $C^*_L(\Delta, G, B)$, which is invertible modulo $C^*_L(\Delta, G, B)$. We define a local index map
$$
\text{ind}_L: KK_0^G( \Delta, B) \to K_0(C^*_L(\Delta, G, B)),
$$
by
$$
\text{ind}_L([H, \pi, F])=\partial([F(t)]),
$$
where $\partial: K_1(M(C^*_L(\Delta, G, B))/C^*_L(\Delta, G, B)) \to K_0(C^*_L(\Delta, G, B))$ is the boundary map on the $K$-theory, and $M(C^*_L(\Delta, G, B))$ is the multiplier algebra of $C^*_L(\Delta, G, B)$.
Similarly, we can define the local index map
$$
\text{ind}_L: KK_1^G( \Delta, B) \to K_1(C^*_L(\Delta, G, B)).
$$

The following result \footnote{In \cite{Qiao-Roe}, Roe and Qiao proved that the local index map is an isomorphism for proper metric spaces.} established the relation between the $K$-homology and the $K$-theory of localization algebras.
\begin{prop}[\cite{Kas_Yu}]\label{thm_local_index_isom}
Let $B$ be any $G$-$C^*$-algebra, and $\Delta$ a finite-dimensional simplicial complex endowed with a $G$-invariant metric.  Then the local index map
$$\text{ind}_L: KK_*^G( \Delta, B) \to K_*(C^*_L(\Delta, G, B)),$$
is an isomorphism.
\end{prop}

\begin{proof}
This result is a consequence of the Mayer--Vietoris sequence and the Five Lemma (see \cite{Yu_localization}).
\end{proof}

By choosing the model of $\mathcal{E}G$ as the union $\bigcup_{d>0}P_d(G)$, the group $KK_*^G(\mathcal{E}G, B)$ is the inductive limit
$$\displaystyle \lim_{d \to \infty} KK_*^G(P_d(G), G, B).$$
The above map induces a local index map
$$\text{ind}_L: KK_*^G(\mathcal{E}G, B) \to \lim_{d\to \infty} K_*(C^*_L(P_d(G), G, B)).$$
It is not difficult to show that the map $\text{ind}_L$ is an isomorphism by Proposition \ref{thm_local_index_isom}.

We will conclude this section by discussing the relation between the Baum--Connes assembly map and the local index map.

For each $d>0$, it is natural to define an evaluation-at-zero map
$$ev: C^*_L(P_d(G), G, B) \to C^*(P_d(G), G, B),$$
given by
$$ev(f)=f(0),$$
for all $f\in C^*_L(P_d(G), G, B)$.
The evaluation-at-zero map induces a homomorphism on $K$-theory
$$ev_*: K_*(C^*_L(P_d(G), G, B)) \to K_*(C^*(P_d(G), G, B))\cong K_*(B\rtimes_{\operatorname{r}}G).$$
Following the argument in \cite{Yu_localization}, it is easy  to check that $\mu=ev_* \circ \text{ind}_L$ holds. In Section 6, we will show that the map $ev_*$ is an isomorphism under some assumptions. Combining the above relation and Proposition \ref{thm_local_index_isom}, it follows that the Baum--Connes assembly map is also an isomorphism under the same assumption.

\section{$C^{\ast}$-algebras associated to infinite dimensional Hilbert spaces}

In this section, we will recall the $C^*$-algebra associated with an infinite-dimensional Hilbert space (defined in \cite{KaHiTr}) and introduce a generalization of this $C^*$-algebra associated with a continuous field of Hilbert spaces due to Tu (see \cite{Tu_BC}).

\subsection{The $C^*$-algebra associated with an infinite-dimensional Euclidean space}\footnote{ We would like to point out that there are other constructions of the $C^*$-algebra $\mathcal{A}(E)$ by Kasparov--Yu in \cite{Kas_Yu} and Gong--Wu--Yu in \cite{GongWuYu}.}
Let $E$ be a separable, infinite-dimensional Euclidean space. Let $E_a$, $E_b$ be any finite-dimensional, affine subspaces of $E$. Let $E_a^0$ be the finite-dimensional linear subspace of $E$
consisting of differences of elements in $E_a$. Let $\mathcal{C}(E_a)$ be the $\mathbb{Z}_2$-graded
$C^*$-algebra of continuous functions from $E_a$ to the complexified Clifford algebra of $E^0_a$ which vanish at infinity. Let $\mathcal{S}$ be the $\mathbb{Z}_2$-graded $C^*$-algebra of all continuous functions on $\mathbb{R}$ vanishing at infinity, where $\mathcal{S}$ is graded according to odd and even functions. Let $\mathcal{A}(E_a)$ be the graded tensor product $\mathcal{S} \widehat{\otimes}\mathcal{C}(E_a)$, where the $\mathbb{Z}_2$-grading on $\mathcal{C}(E_a)$ is induced from $\text{Cliff}(E_a^0)$.

Assume $E_a \subset E_b$. There exists a decomposition $E_b=E_{ba} \oplus E_a$, where $E_{ba}$ is the orthogonal complement of $E_a$ in $E_b$. For each element $v_b \in E_b$, there exists a unique decomposition $v_b=v_{ba}+v_a$, for some $v_{ba} \in E_{ba}$, $v_a \in E_a$.

For each function $h \in \mathcal{C}(E_a)$, we can extend it to a function on $E_b$ via $\tilde{h}(v_b)=h(v_a)$, for all $v_b=v_{ba}+v_a$. The decomposition $E_b=E_{ba} \oplus E_a$ gives rise to a Clifford algebra valued function, denoted by $C_{ba}: E_b \to\mbox{ Cliff}(E^0_b)$ on $E_b$ which maps $v_b \in E_b$ to $v_{ba} \in E_{ba} \subset \mbox{Cliff}(E_b^0)$.

Denote by $X: \mathcal{S} \to \mathcal{S}$ the function of multiplication by $ x$ on $\mathbb{R}$, viewed as a degree one, essentially selfadjoint, unbounded multiplier of $\mathcal{S}$ with domain the compactly supported functions in $\mathcal{S}$.

 \begin{defn}[\cite{GHT}]\leavevmode
\begin{enumerate}
\item Let $E_a \subset E_b$ be a pair of finite-dimensional affine subspaces of $E$. One can define a homomorphism
$$
\beta_{ba}: \mathcal{A}(E_a) \to \mathcal{A}(E_b)
$$
by $\beta_{ba}(f \widehat{\otimes}h)=f(X\widehat{\otimes }1+ 1\widehat{\otimes}C_{ba}) (1 \widehat{\otimes} \tilde{h})$, for all $f \in \mathcal{S}$, $h \in \mathcal{C}(E_a)$.
\item We define a $C^*$-algebra
$$\mathcal{A}(E):=\varinjlim \mathcal{A}(E_a),$$
where the direct limit is  over all finite-dimensional affine subspaces.
 \end{enumerate}
\end{defn}

\begin{rem}
If $E_a \subset E_b \subset E_c$, then we have $\beta_{cb}\circ \beta_{ba}=\beta_{ca} $, therefore the direct limit is well-defined.
\end{rem}

Given any discrete group $\Gamma$, if $\Gamma$ acts on the Euclidean space $E$ by linear isometries, then the $\Gamma$-action on $E$ induces a $\Gamma$-action on the $C^*$-algebra $\mathcal{A}(E)$.
Note that $\mathcal{A}(\{0\}) = \mathcal{S}$. For each $f \in \mathcal{S}$, let $\beta_t(f)=f_t(X\widehat{\otimes}1 + 1 \widehat{\otimes}C)$ for every $t \in [1, \infty)$, where $f_t(x)=f(x/t)$.

We define the Bott map
$$\beta_*:K_*(\mathcal{S}\rtimes_{max} \Gamma) \to K_*(\mathcal{A}(E)\rtimes_{max} \Gamma).
$$
to be the homomorphism induced by the asymptotic morphism
$$
\beta_t:\mathcal{S}\rtimes_{max} \Gamma \to \mathcal{A}(E)\rtimes_{max} \Gamma,$$
given by $f \mapsto \beta_t(f) $, for each $t \in [1, \infty)$.
The following result is due to Higson, Kasparov and Trout \cite{GHT}.

\begin{thm}[Infinite-dimensional Bott Periodicity \cite{KaHiTr}]
Let $\Gamma$ be a countable discrete group, $E$ an infinite-dimensional Euclidean space with a $\Gamma$-action by linear isometries. Then the Bott map
$$\beta_*:K_*(\mathcal{S}\rtimes_{\operatorname{max}} \Gamma) \to K_*(\mathcal{A}(E)\rtimes_{\operatorname{max}} \Gamma)$$
is an isomorphism.
\end{thm}

\subsection{Generalization to the field cases} In the rest of this section, we will generalize the construction of Higson--Kasparov--Trout to the case of the continuous field. The following construction is essentially due to Tu \cite{Tu_BC}. The usage of the probability space is essentially due to Higson \cite{Hig_bivariant}.

Let $\Gamma$ be a countable, discrete group with identity element $e \in \Gamma$ , and let $X$ be a compact Hausdorff space admitting a $\Gamma$-action by homeomorphisms. Let us recall the definition of the transformation groupoid, denoted by $X \rtimes \Gamma$, associated with a group action $\Gamma \curvearrowright X$.

\begin{defn} As a topological space, $X \rtimes \Gamma=\{(x,g): x \in X, g \in \Gamma\}$ is equipped with the product topology. In addition, the topological space is endowed with the following structure.

\begin{enumerate}
\item The product is given by $(x,g)(x',g')=(x,gg')$, for all $(x,g), (x',g') \in X \rtimes \Gamma$ satisfying $x'=xg$.
\item The inverse is given by $(x,g)^{-1}=(xg, g^{-1})$, for all $(x, g) \in X \rtimes \Gamma$.
\end{enumerate}
\end{defn}

Let us recall the definition of a continuous field of Hilbert spaces over a compact space. Let $X$ be a compact topological space, and let $\big(\mathcal{H}_x\big)_{x \in X}$ be a family of Banach spaces. Denote $\mathcal{H}=\bigsqcup_{x \in X}\mathcal{H}_{x}$. Let $\Theta(X, \mathcal{H})$ be a collection of sections $s: X \to \mathcal{H}$ satisfying $s(x) \in \mathcal{H}_{x}$, for all $x \in X$.

\begin{defn}\label{defn_continuous_field}
Let $X$ be
a compact space. A continuous field of Banach spaces over $X$ is a family of Banach spaces $\big(\mathcal{H}_x\big)_{x \in X}$, with a set of sections $\Theta(X, \mathcal{H})$, such that
\begin{enumerate}
  \item the set of sections $\Theta(X, \mathcal{H})$ is a linear subspace of the direct product $\prod_{x \in X} \mathcal{H}_x$,
  \item for every $x \in X$, the set of all $s(x)$ for all $s \in \Theta(X, \mathcal{H})$ is dense in $\mathcal{H}_x$,
  \item for every $s \in \Theta(X, \mathcal{H})$, the function $x \mapsto \|s(x)\|$ is a continuous function on $X$,
  \item let $s: X \to \mathcal{H}$ be a section, i.e. $s(x) \in \mathcal{H}_{x}$, for all $x \in X$. If for every $x \in X$, and every $\epsilon>0$, there exists a section $s' \in \Theta(X, \mathcal{H})$ such that $\|s(y)-s'(y)\|< \epsilon$ for all $y$ in some neighborhood of $x$, then $s \in \Theta(X, \mathcal{H})$.
\end{enumerate}
\end{defn}

If every fiber $\mathcal{H}_x$ is a Hilbert space, we will say $\big(\mathcal{H}_x\big)_{x \in X}$ is a continuous field of Hilbert spaces. If every fiber is a $C^*$-algebra and the collection of sections is closed under the $*$-operation and the multiplication, the continuous field is called a continuous field of $C^*$-algebras.

Let $X \rtimes \Gamma$ be a transformation groupoid associated with the right group action $\Gamma \curvearrowright X$, and let $\big(\mathcal{H}_x\big)_{x \in X}$ be a continuous field of Hilbert spaces over $X$. Let us recall the concept of the affine isometric action of $X \rtimes \Gamma$ on  the continuous field of Hilbert spaces $\big(\mathcal{H}_x\big)_{x \in X}$.

\begin{defn}
Let $\big( \mathcal{H}_x \big)_{x \in X}$ be a continuous field of Hilbert spaces over $X$. We say that the transformation groupoid $X \rtimes\Gamma$ acts on $\big( \mathcal{H}_x \big)_{x \in X}$ by affine isometries if
\begin{enumerate}
\item for each $(x,g)$, there exists an affine isometric map $V_{(x,g)}:\mathcal{H}_{xg} \to \mathcal{H}_{x}$;
\item $V_{(x,g)}V_{(xg,h)}=V_{(x,gh)}$, for all $x \in X$, $g,h \in \Gamma$;
\item $V_{(x,e)}=id_{H_x}$.
\item for each continuous section $s \in \Theta(X, \mathcal{H})$ and each $g \in \Gamma$, the section $x \mapsto V_{(x, g)}(s(xg))$ is a continuous section.
\end{enumerate}
\end{defn}

An affine isometry $V_{(x,g)}:\mathcal{H}_{xg} \to \mathcal{H}_{x}$ is an isometry of the form $$V_{(x, g)}(v)=U_{(x, g)}(v)+b(x, g),$$ for all $v \in \mathcal{H}_{xg}$, where  $U_{(x,g)}:\mathcal{H}_{xg} \to \mathcal{H}_{x}$ is a unitary map, $b(x, g)$ is a vector in $ \mathcal{H}_x$. The above condition (2) implies that $b(x, gh)=U_{(x, g)}(b(xg, h))+b(x, g)$, for all $x \in X$, $g, h \in \Gamma$.

Since each fiber of the continuous field is a Hilbert space, we can define a $C^*$-algebra $\mathcal{A}(\mathcal{H}_x)$ associated with each fiber $\mathcal{H}_x$. Then we obtain a bundle of $C^*$-algebras $\big(\mathcal{A}(\mathcal{H}_x)\big)_{x \in X}$. Next, we will introduce a structure of a continuous field of $C^*$-algebras for the bundle $\big(\mathcal{A}(\mathcal{H}_x)\big)_{x \in X}$, some more details can also be found in \cite{HiGu}.

 A function $s: X \to \bigsqcup_{x \in X}\mathcal{A}(\mathcal{H}_x)$ is said to be a continuous section, if it satisfies

\begin{enumerate}
\item $s(x) \in \mathcal{A}(\mathcal{H}_x)$, for all $x \in X$,

\item for each $x \in X$, $\epsilon>0$, there exists a neighborhood $x \in U \subset X$, such that for each $y \in U$, there is a linearly isometric embedding $\phi_y: \mathbb{R}^n \to \mathcal{H}_y$ satisfying that
\begin{itemize}
\item  for each $1 \leq i \leq n$, the map $y \mapsto \phi_y(e_i)$ is a local continuous section over $U$, where $\{e_i\}_{i=1}^n$ is an orthonormal basis of $\mathbb{R}^n$,
\item there is an element $v \in \mathcal{A}(\mathbb{R}^n)$, such that
   $$\|(\phi_y)_*(v)-s(y)\|_y < \epsilon,$$
    for all $y \in U$, where the $(\phi_y)_*:  \mathcal{A}(\mathbb{R}^n)\to \mathcal{A}(\mathcal{H}_y)$ is induced by the linear embedding $\phi_y: \mathbb{R}^n \to \mathcal{H}_y$.
\end{itemize}
\end{enumerate}
The above definition gives rise to a continuous field structure on $\big(\mathcal{A}(\mathcal{H}_x)\big)_{x \in X}$. For any $x \in X$, note that there exists a collection of sections $e_n \in \Theta(X, \mathcal{H}) $, such that $\{e_n(x)\}_{n \in \mathbb{N}}$ is an orthonormal basis of the fiber bundle  $\mathcal{H}_x$. Indeed, let $\{a_n \in \Theta(X, \mathcal{H})\}_{n \in\mathbb{N}}$ be the collection of continuous sections such that $\{a_n(x)\}$ is a basis for $\mathcal{H}_x$. For each positive integer $n$, there exists a neighborhood $x \in U'_n$, such that the collection $\{a_i(y): 1\leq i \leq n\}$ is linearly independent for each $y \in U'_n$. By the Gram--Schmidt process, one can find a family of continuous local sections $\{e'_i: 1\leq i \leq n\}$ over $U'_n$, such that the collection of vectors $\{e'_i(y): 1 \leq i \leq n\}$ are mutually orthogonal with norm one for each $y \in U'_n$. By Urysohn Lemma, we can find a open set $U_n \subset U'_n$ and a function $f: X \to [0, 1]$, such that $\text{supp}(f) \subset U'_n$ and $f(y)=1$ for each $y \in U_n$. For each $i$, we obtain a global section $e_i$ by setting $e_i(y)=f(y)e'_n(y)$ for all $y \in U'_n$ and extending by zero outside $U'_n$. By induction on $n$, we obtain such a collection of continuous global sections $\{e_n\}_{n \in \mathbb{N}}$.

The $C^*$-algebra of all continuous sections of the continuous field of $C^*$-algebras is denoted by $\mathcal{A}(X, \mathcal{H})$.

Define a $\Gamma$-action on $\mathcal{A}(X, \mathcal{H})$ as follows. For each $\gamma \in \Gamma$, we have an isometry $V_{(x, \gamma)}: \mathcal{H}_{x\gamma}\to \mathcal{H}_x$ given by
$$V_{(x, \gamma)}(v)=U_{(x, \gamma)}(v)+b(x, \gamma),$$
for all $v \in \mathcal{H}_{x\gamma}$ and $x \in X$. A local continuous  affine distribution over an open subset $U\subset X$ is a collection  of affine subspaces $\{E_a(y)\}_{y \in U}$, such that there exist a collection of continuous local sections $\{s, v_1, v_2, \cdots, v_k\}$ over $U$, for which $E_a(y)= s(y)+\text{span}\{v_1(y), v_2(y), \cdots , v_k(y)\}$ for each $y \in U$. The associated linear subspace consisting of differences of elements in $E_a(y)$ is denoted by $E^0_a(y)=\text{span}\{v_1(y), v_2(y), \cdots, v_k(y)\}$. We have a continuous local distribution of $C^*$-algebras $\big(\text{Cliff}(E^0_a(y)\big)_{y \in U}$. Because $(V_{(x\gamma^{-1}, \gamma)}\big)_{x \in X}$ is a collection of continuous isometries for each $\gamma \in \Gamma$ , we obtain another continuous local affine distribution $\{V_{(y\gamma^{-1}, \gamma)}(E_a(y))\}_{y \in U}$ over $\gamma(U)$ with associated linear distribution $\{U_{(y\gamma^{-1}, \gamma)}(E_a(y))\}_{y\in U}$. The unitary $U_{(y, \gamma^{-1})}: H_{y\gamma^{-1}}\to H_y$ induce a homomorphism
$$U_{(y, \gamma^{-1})}: \text{Cliff}(U_{(y\gamma^{-1}, \gamma)}(E^0_a(y))) \to \text{Cliff}(E^0_a(y)),$$
for all $y \in U$.

Thus, for every $\gamma \in \Gamma$, we get a homomorphism
$$\gamma: \mathcal{A}(E_a(y)) \to \mathcal{A}(V_{(y \gamma^{-1}, \gamma)}E_a(y)),$$
by
$$\gamma\cdot (f \widehat{\otimes}h)=f \widehat{\otimes} \gamma \cdot h,$$
for all $f \in \mathcal{S}$, $h: E_a(y) \to \text{Cliff}(E_a^0(y))$, where $(\gamma \cdot h)(v)=U_{(y\gamma^{-1}, \gamma)}(h(V_{(y, \gamma)}(v)))$.

Let $\{E_b(y)\}_{y \in U}$ be another continuous affine distribution with $E_a(y) \subset E_b(y)$, for all $y \in U$. Then, there exists a linear continuous distribution $\{E_{ba}(y)\}_{y \in U}$ such that $E_b(y)=E_{ba}(y) \oplus E_a(y)$ for all $y \in U$.

Let $E_a=\bigsqcup_{y \in U}E_a(y)$. Define $\mathcal{A}(U, E_a)$ to be the algebra of all the bounded continuous sections $s: U \to \mathcal{A}(E_a)$ with $s(y) \in \bigsqcup_{y \in U}\mathcal{A}(E_a(y))$ for all $y \in U$. For all $\gamma \in \Gamma$, denote $\gamma E_a(y)=V_{(y\gamma^{-1}, \gamma)}(E_a(y))$,  $U_{\gamma} E_a=\bigsqcup_{y \in U} U_{(y, \gamma^{-1},\gamma)} E_a(y)$ and $\gamma E_a=\bigsqcup_{y \in U} \gamma E_a(y)$.

\begin{lem}\label{lem,diagr_for_action}
The following diagram
\begin{equation*}
\xymatrix{
\mathcal{A}(U, E_a)\ar[d]^{\gamma}\ar[r]^{\beta_{U, ba}}& \mathcal{A}(U, E_b)\ar[d]^{\gamma}\\
	\mathcal{A}(\gamma U, \gamma E_a)\ar[r]^{\beta^{\gamma}_{\gamma U, ba}}& \mathcal{A}(\gamma U, \gamma E_b).}
\end{equation*}
is commutative, where the maps $\beta_{U, ba}$ and $\beta^{\gamma}_{\gamma U, ba}$ are defined fiber-wise.
\end{lem}

\begin{proof}
Since $\mathcal{S}$ is generated by $g_0(e)=e^{-x^2}$ and $g_1(x)=xe^{-x^2}$ by the Stone-Weierstrass theorem, it suffices to show that
$$\gamma(\beta_{U, ba}(g \widehat{\otimes}h))=\beta^{\gamma}_{\gamma U, ba}(\gamma(g \widehat{\otimes} h)),$$
for $g$ equal to $g_0$ or $g_1$ and $h: U \to \text{Cliff}(E_a)$ a continuous local section over $U$. We will prove the lemma for $g=g_0$; the case for $g=g_1$ can be proved similarly.

Let $C_{ba}$ and $C_{ba}^{\gamma}$ be the Clifford multiplications on $\mathcal{C}(E_a)$ and $\mathcal{C}(\gamma E_a)$ under the transformation $\varphi_{a, y}:E_a(y) \to \gamma E_a(y)$ for each $y \in U$.
Since we have $E_b=E_{ba}\oplus E_a$, and $\gamma(E_b)=U_{\gamma}(E_{ba})+\gamma E_a$, then
\begin{align*}
\beta_{U, ba}^{\gamma}(\gamma(g_0\widehat{\otimes}h)) &=\beta_{U, ba}^{\gamma}((g_0\widehat{\otimes}\gamma(h)))\\
&=g_0(X \widehat{\otimes}1 +1 \widehat{\otimes}C^{\gamma}_{ba})(1 \widehat{\otimes}(\widetilde{\gamma h}))\\
&=g_0(x)\widehat{\otimes}g_0(U_{\gamma}v_{ba})\gamma h(v_a),
\end{align*}
and
$$\gamma(\beta_{U, ba}(g_0 \widehat{\otimes}h) )=\gamma (g_0(x)\widehat{\otimes} g_0(\|v_{ba}\|) h(v_a))=g_0(x)\widehat{\otimes}g_0(\|v_{ba}\|)\gamma h(v_a),
$$
where $v_b$, $v_a$, $v_{ba}$ are continuous local sections over $U$ with $v_b(y)=v_{ba}(y)+ v_a(y)$ for all $y \in U$, and $\gamma(v_b)=U_{\gamma}(v_{ba})+\gamma (v_a)$. Since $U_{\gamma}$ is unitary,
thus
$$\gamma(\beta_{U, ba}(g \widehat{\otimes}h))=\beta^{\gamma}_{\gamma U, ba}(\gamma(g \widehat{\otimes} h)),$$
for all $g=g_0$.
\end{proof}

Because of the definition of $\mathcal{A}(X, \mathcal{H})$,  and the above commutative diagram, we have a $\Gamma$-action on $\mathcal{A}(X, \mathcal{H})$.

Define a fiber-wise Bott map $\beta_t: C(X)\widehat{\otimes} \mathcal{S} \to \mathcal{A}(X, \mathcal{H}) $  as follows. Viewing an element in $C(X) \widehat{\otimes} \mathcal{S}$ as a continuous function $f: X \to \mathcal{S}$, we obtain an element $\beta_t(f_x) \in \mathcal{A}(\mathcal{H}_x)$ for each $x \in X$. By the definition of the continuous field structure, it is easy to check that $\big(\beta_t(f_x)\big)_{x \in X} \in \mathcal{A}(X,  \mathcal{H})$. Thus, we have an asymptotic morphism
\begin{align*}
\beta_t:C(X)\widehat{\otimes}\mathcal{S} &\to \mathcal{A}(X, \mathcal{H})\\
 \big(f_x\big)_{x \in X}& \mapsto \big(\beta_t(f_x)\big)_{x \in X},\\
\end{align*}
for all $t \in [1, \infty)$.

We have an affine isometric action of $\Gamma$ on the continuous field of Hilbert spaces $\big(\mathcal{H}_x\big)_{x \in X}$, and it is easy to check that $(\beta_t)_{t \in [1, \infty)}$ is an asymptotic $\Gamma$-equivariant morphism. This asymptotic morphism induces a map on $K$-theory of the reduced crossed products
$$\beta_*: K_*((C(X) \otimes\mathcal{S})\rtimes_{\operatorname{r}} \Gamma_0)\to K_*(\mathcal{A}(X, \mathcal{H}) \rtimes_{\operatorname{r}} \Gamma_0),$$
for every finite subgroup $\Gamma_0\leq \Gamma$.
Following the argument in \cite{KaHiTr}, we also have a fiber-wise defined Dirac map. Let us briefly recall the definition of the Dirac map on each fiber $\mathcal{A}(\mathcal{H}_x)$ for each $x \in X$.

Let $E_a(x) \subset \mathcal{H}_x$ be a finite-dimensional affine subspace. Define $V_a(x)$ to be the Hilbert space of square integrable functions from $E_a(x)$ into $\text{Cliff}(E_a^0(x))$, where $E_a^0(x)$ is the linear space of differences between pairs of vectors in $E_a(x)$, and the norm on $\text{Cliff}(E_a^0(x))$ is obtained by fixing an orthonormal basis on $E_a^0(x)$. If $E_a(x) \subset E_b(x)$, then there is a canonical isomorphism
$$V_b(x)\cong V_{ba}(x) \widehat{\otimes} V_a(x)$$
where $V_{ba}(x)$ is the Hilbert space associated with the orthogonal complement $E_{ba}^0(x)$ of $E_a^0(x)$ in $E_b^0(x)$.

We define a unit vector $\xi_0\in V_{ba}(x)$ by
$$\xi_0(v_{ba})=\pi^{-\text{dim}(E^0_{ba}(x))/4} \exp(-\frac{1}{2}\|v_{ba}\|^2),$$
for all $v_{ba} \in E^0_{ba}$. Regarding $V_{a}(x)$ as a subspace of $V_b(x)$ via the isometry $\xi \mapsto \xi_0\widehat{\otimes} \xi$, we define

$$V(x)=\varinjlim V_a(x).$$

Using the similar method of construction of the continuous field structure on $\big(\mathcal{A}(\mathcal{H}_x)\big)_{x \in X}$, we obtain a continuous field structure on $\big(\mathcal{K}(V(x))\big)_{x \in X}$, where $\mathcal{K}(V(x))$ is the algebra of all compact operator on $V(x)$ for each $x \in X$. Let $V =\bigsqcup_{x \in X} V(x)$. Define $\mathcal{K}(X, V)$ to be the $C^*$-algebra of all continuous sections of the continuous field $\big(\mathcal{K}(V(x))\big)_{x \in X}$. By the structure of the continuous field of $\big(\mathcal{K}(V(x))\big)_{x \in X}$, we have that the $K$-theory of the $C^*$-algebra $C(X)$ is the same as the $K$-theory of $\mathcal{K}(X, V)$.

Denote by $\mathfrak{s}(x)=\varinjlim\mathfrak{s}_a(x)$ the direct limit of the Schwartz subspaces $\mathfrak{s}_a(x) \subset V_a(x)$. If $E_a(x)$ is a finite-dimensional affine subspace, then the Dirac operator $D_a(x)$ is defined by
$$D_a(x) \xi=\sum_{i=1}^{n} (-1)^{deg(\xi)}\frac{\partial (\xi)}{\partial x_i}v_i,$$
for every homogeneous element $\xi \in \mathfrak{s}(E_a(x))$,
where $\left\{v_1, v_2, \ldots, v_n\right\}$ is an orthonormal basis for $E^0_a(x)$, and $\left\{x_1, x_2, \ldots, x_n\right\}$ are the dual coordinates to $\left\{v_1, v_2, \ldots, v_n\right\}$. The Clifford operator on $E_a(a)$ is given by

$$C_a(x)\xi=\sum_{i=1}^{n}x_iv_i\xi.$$

For each fiber $\mathcal{H}_x$, choose the dense subset $E(x):=\left\{s(x): s \in \Theta(X, \mathcal{H}) \right\}$ where $\Theta(X, \mathcal{H})$ is the space of sections in Definition \ref{defn_continuous_field}. Fix a direct sum decomposition
$$E(x)=E_0(x)\oplus E_1(x)\oplus E_2(x) \oplus \cdots,$$
where each $E_i(x)$ is a finite-dimensional linear subspace of $E(x)$. For each $n$, we define an unbounded operator $B_{n, t}$ on $V(x)=\varinjlim V_n(x)$,  by the formula

$$
B_{n, t}=t_0D_0(x)+t_1D_1(x)+\cdots+t_{n-1}D_{n-1}(x)+t_n(D_n(x)+C_n(x))+\cdots
$$
where $t_{i}=1+i/t$, and $V_n(x)$ is the Hilbert space of square integrable functions from $E_n(x)$ to $\text{Cliff}(E_n(x))$. This infinite sum is well-defined since any vector in the Schwartz space $\mathfrak{s}(x)$ can be approximated by the one which has only finitely many nonzero terms in its infinite series. It is well-known that the operators $B_{n, t}(x)$ are essentially selfadjoint. Following the argument in \cite{KaHiTr}, we obtain an asymptotic morphism $\alpha^n$ from $\mathcal{A}(E_0(x)\oplus E_1(x)\oplus\cdots \oplus E_n(x))$ to $\mathcal{S}\widehat{\otimes}\mathcal{K}(V(x))$ by
$$\alpha^n_{t}(x)(f \widehat{\otimes} h)=f_t(X \widehat{\otimes}1+ 1 \widehat{\otimes}B_{n, t})(1 \widehat{\otimes}M_{h_t}),$$
for each $f \widehat{\otimes} h \in \mathcal{A}(E_0(x)\oplus E_1(x)\oplus\cdots \oplus E_n(x))$ where $h_t(v)=h(v/t)$ for all $t \in [1, \infty)$, $v \in E_0(x)\oplus E_1(x)\oplus\cdots \oplus E_n(x)$ and $M_{h_t}$ is the operator of left multiplication by the function $h_t$. Moreover, the diagram

\begin{tikzcd}
\mathcal{A}(E_0(x)\oplus E_1(x)\oplus\cdots \oplus E_n(x)) \arrow[r, "\alpha^n(x)"] \arrow[d,"\beta" ]
& \mathcal{S}\widehat{\otimes}\mathcal{K}(V(x))  \arrow[d, "=" ] \\
\mathcal{A}(E_0(x)\oplus E_1(x)\oplus\cdots \oplus E_{n+1}(x)) \arrow[r,  "\alpha^{n+1}(x)" ]
&\mathcal{S}\widehat{\otimes} \mathcal{K}(V(x))
\end{tikzcd}\\
is asymptotically commutative. As a result, we get an asymptotic morphism $\alpha(x): \mathcal{A}(\mathcal{H}_x) \to \mathcal{S} \widehat{\otimes}\mathcal{K}(V(x))$. Moreover, we obtain an asymptotic morphism
$$\alpha: \mathcal{A}(X, \mathcal{H}) \to \mathcal{S} \widehat{\otimes} \mathcal{K}(X, V).$$

Following the argument in \cite{KaHiTr}, the map induced by $\alpha$ on $K$-theory is the inverse map of the fiber-wise defined Bott map. Consequently, we have the following result.

\begin{thm}\label{thm,eq_Bott_isom_fini_gro}
Let $\Gamma$ be a countable discrete group and $X$ a compact Hausdorff space with $\Gamma$-action. Assume the associated transformation groupoid $X \rtimes \Gamma $ acts on a continuous field of Hilbert spaces $\big(\mathcal{H}_x\big)_{x \in X}$ by affine isometries. Then for each finite subgroup $\Gamma_0\leq \Gamma$, the Bott map
$$\beta_*: K_*((C(X)\otimes\mathcal{S})\rtimes_{\operatorname{r}} \Gamma_0)\to K_*(\mathcal{A}(X, \mathcal{H}) \rtimes_{\operatorname{r}} \Gamma_0)$$
induces an isomorphism on $K$-theory.
\end{thm}

An affine isometric action of a transformation groupoid $X \rtimes \Gamma$ on a continuous field of Hilbert spaces $\big( \mathcal{H}_x \big)_{x \in X}$ is said to be proper, if for any $R > 0$, the set $\{g \in \Gamma: \exists x \in X ~\text{such that}~ V_{(x,g)}(B(xg, R)) \cap B(x,R)\neq \emptyset\}$ is finite, where $B(x,R)$ is the set of all elements in $H_{x}$ with norm less than $R$. Due to Tu (see \cite{Tu_BC}), the a-T-menability of the transformation groupoid guarantees the existence of a proper affine isometric action on a continuous field of Hilbert spaces $\big(\mathcal{H}_x\big)_{x \in X}$.

\begin{defn}
Let $X \rtimes \Gamma$ be a transformation groupoid. A continuous function $\varphi: X \rtimes \Gamma \to \mathbb{R}$ is said to be conditionally negative definite if

\begin{enumerate}\label{CND}
\item $\varphi(x,e)=0$, for all $x \in X$,

\item $\varphi(x, g)=\varphi(xg, g^{-1})$,

\item $\sum_{i,j=1}^n t_it_j \varphi(xg_i, g_i^{-1}g_j) \leq 0$, for all $\{t_i\}_{i=1}^n \subset \mathbb{R}$ with $\sum_{i=1}^n t_i=0$, $g_i \in \Gamma$, and $x \in X$.
\end{enumerate}

\end{defn}

A conditionally negative definite function $\varphi: X \rtimes \Gamma \to \mathbb{R}$ is said to be proper if for any $R>0$, the number of elements in the set $\left\{g \in \Gamma: \exists x \in X,~\text{such that}~ |\varphi(x, g)|\leq R\right\}$ is finite.

The concept of a-T-menability for groupoids was introduced by Tu in \cite{Tu_BC}.
\begin{defn}
A transformation groupoid $X \rtimes \Gamma$ is said to be a-T-menable if there exists a proper conditionally negative definite function $\varphi: X \rtimes \Gamma \to \mathbb{R}$.
\end{defn}

Now, let us recall the construction of the transformation groupoid from a coarsely embeddable group, by Skandalis, Tu, and Yu in \cite{Skan_Tu_Yu}.
\begin{prop}[\cite{Skan_Tu_Yu}]
Let $h: \Gamma \to \mathcal{H}_0$ be a coarse embedding. Then there exists a compact Hausdorff space $X$ with a $\Gamma$-action, such that
\begin{enumerate}
\item for any finite subgroup $\Gamma_0\leq \Gamma$, $X$ is $\Gamma_0$-contractible,
\item the groupoid $X \rtimes \Gamma$ has a proper continuous conditionally negative definite function.
\end{enumerate}
\end{prop}

Let us describe the construction of the topological space $X$. For any fixed element $\gamma \in \Gamma$, we define a bounded function $f_{\gamma}:\Gamma \to \mathbb{R}$ by
 $$f_{\gamma}(y)=\|h(y)-h(y\gamma)\|,$$
for any $y \in \Gamma$.

Let $c_0(\Gamma)$ be the $C^*$-subalgebra of $\ell^{\infty}(\Gamma)$ consisting all functions vanishing at infinity. We define a $\Gamma$-action on $\ell^{\infty}(\Gamma)$ by $(\gamma \cdot f)(x)=f(x \gamma)$, for each $f\in \ell^{\infty}(\Gamma)$, $x, \gamma \in \Gamma$.

Let $X'$ be the spectrum of the unital commutative $\Gamma$-invariant $C^*$-subalgebra of $\ell^{\infty}(\Gamma)$ generated by all constant functions, $c_0(\Gamma)$ functions, and all functions of the form $f_{\gamma}$ together with their translations by group elements in $G$. It is obvious that every function $f_{\gamma}$ extends continuously to $X'$. Indeed, the space $X'$ is a compactification of $\Gamma$, and it admits a right action of $\Gamma$ induced by the $\Gamma$-action on $C(X')$ where $C(X')$ is viewed as a $C^*$-subalgebra of $\ell^{\infty}(\Gamma)$.

We obtain a continuous conditionally negative definite function defined on the transformation groupoid $X' \rtimes \Gamma$ by $\varphi'(y,\gamma)=f_{\gamma}(y)$.

Let $X$ be the probability space of $X'$. It is a second countable, compact space equipped with the weak-$*$ topology (A reference of weak-$*$ topology is Chapter 1 in \cite{DouglasBook}) and it admits a $\Gamma$-action induces by the action of $\Gamma$ on $X'$. We define a conditionally negative definite function on $X\rtimes \Gamma$ by
$$\varphi(m,\gamma)=\int_{X'}\varphi'(y, \gamma) d m(y)$$
for any $m \in X$.

\begin{prop}\label{prop_proper_CND_on_groupoid}
Let $\Gamma \to \mathcal{H}_0$ be the coarse embedding as above. The continuous map $\varphi: X \rtimes G \to \mathbb{R}$  defined above is a proper conditionally negative definite function.
\end{prop}

\begin{proof}
It is obvious that condition (1) in the Definition  \ref{CND} is satisfied. Let us verify condition (2). For each $(x, g) \in X \rtimes \Gamma$, we have
\begin{align*}
\varphi(mg, g^{-1}) &= \int_{X'}\varphi'(y, g^{-1}) d (mg)\\
                                       &= \int_{X'}\varphi'(yg, g^{-1}) d m\\
                                        & = \int_{X'} \varphi'(y, g) d m\\
                                       &= \varphi(m, g).
\end{align*}
The third equality follows from $\varphi'(yg, g^{-1})=\varphi(y, g)$ for all $(y, g) \in X \rtimes \Gamma$. Condition (3) follows from the fact that $\sum_{i,j=1}^n t_it_j \varphi'(yg_i, g_i^{-1}g_j) \leq 0$, for all $\{t_i\}_{i=1}^n \subset \mathbb{R}$ with $\sum_{i=1}^n t_i=0$, $g_i \in \Gamma$, and $y \in X'$.

The properness of $\varphi$ follows from the definition of $\varphi'$ and the fact that the map $h: \Gamma \to \mathcal{H}_0$ is a coarse embedding.
\end{proof}

The space $X \times \Gamma$ is equipped with the product topology. Let $C_c(X \times \Gamma)$ be the $C^*$-algebra of all complex valued functions on $X \times \Gamma$ with compact support. Define
$$
C^0_{c}(X \rtimes \Gamma):=\left\{f \in C_c(X \times \Gamma): \sum_{g \in \Gamma}f(x,g)=0\right\}.
$$

Let $\varphi: X \rtimes \Gamma \to \mathbb{R}$ be a continuous, proper conditional negative definite
function. Then we can define a continuous field of Hilbert spaces as follows.

For each $x \in X$, consider a linear space $C^0_{c}(\Gamma):=\left\{f \in C_c(\Gamma): \sum_{g \in \Gamma}f(g)=0\right\}$, and define a sesquilinear form
$$
\left\langle\xi,\eta\right\rangle_x=-\frac{1}{2}\sum_{g, g' \in \Gamma}\xi(g) \widebar{\eta(g')}\varphi(xg^{-1}, gg'),
$$
for all $\xi, \eta \in C^0_{c}(\Gamma)$. Since $\varphi$ is conditionally negative definite type, the form above turns out to be positive semidefinite and one can quotient out by the zero subspace, denoted by $E_x$.
Then complete $E_x$ to a Hilbert space, denoted by $\mathcal{H}_x$.
For any function $f\in C^0_c(X \times \Gamma)$, we can view it as a continuous map $\xi: X \to C^0_c(\Gamma)$ by $\xi(x)=f(x, \cdot)\in C^0_c(\Gamma)$.

Let us introduce a continuous field structure on the collection of Hilbert spaces $\big(H_x\big)_{x \in X}$. It suffices to define the space of the continuous sections, denoted by $\Theta(X, \mathcal{H})$, where $\mathcal{H}=\bigsqcup_{x \in X} \mathcal{H}_x$. A map $\xi: X \to \mathcal{H} $ is called a continuous section, if it satisfies

\begin{enumerate}
\item $\xi(x) \in \mathcal{H}_x$, for every $x \in X$,

\item $\forall$ $x \in X$, $\forall$ $\epsilon >0$, there exists an element $\xi' \in C^0_{c}(X \times \Gamma)$, such that $\|\xi(y)-\xi'(y)\|_{\mathcal{H}_y} < \epsilon$ for all $y$ in some neighborhood of $x$.
\end{enumerate}

The affine isometric action of $X \rtimes \Gamma$ is defined as follows. For every $\gamma \in \Gamma$, and every $x \in X$, the unitary $U_{(x, \gamma)}: \mathcal{H}_{x \gamma} \to \mathcal{H}_x$ is defined by $U_{(x, \gamma)}(f)(g)=f(\gamma^{-1}g)$
for all $f \in E_x$, and for all $g \in \Gamma$, $\gamma\in \Gamma$, then extends to a unitary $U_{(x, \gamma)}: \mathcal{H}_{x\gamma} \to \mathcal{H}_x$.
The cocycle $b(x, g) $ is defined to be the element in $E_x$ represented by the function $\delta_g-\delta_e$. Let $V_{(x, \gamma)}(v)=U_{(x, \gamma)}(v)+b(x, \gamma)$ for all $v \in \mathcal{H}_{x \gamma}$, $(x, \gamma) \in X \rtimes \Gamma$.
It is easy to check that the collection of affine isometries $\big(V_{(x, \gamma)}\big)_{(x, \gamma)\in X \rtimes\Gamma}$ consists of a proper affine isometric action of $X \rtimes \Gamma$ on the continuous field of Hilbert spaces $\big(\mathcal{H}_x\big)_{x \in X}$.

\section{Twisted Roe algebras and twisted localization algebras}

Let $1 \to N \to G \to G/N \to 1 $ be a short exact sequence of countable discrete groups. In this section we will construct twisted Roe algebras and twisted localization algebras with coefficients in some $G$-$C^*$-algebra, and prove that the twisted Baum--Connes conjecture with coefficients holds for the group $G$, under the assumption that both $G$ and $G/N$ are coarsely embeddable into Hilbert spaces.

Fix a left invariant proper metric on $G$. This metric restricts to every subgroup of $G$ and the quotient group $G/N$ is endowed with the quotient metric.

\subsection{Some Geometric Constructions} In this subsection, we will construct a compact topological $G$-space $Y$, such that
\begin{enumerate}
\item for every subgroup $N' \leq G$ containing $N$ with finite index, i.e., $|N'/N| < \infty$, the transformation groupoid $X\rtimes N'$ is a-T-menable,
\item for every finite subgroup $G_0\leq G$, the space $Y$ is $G_0$-contractible.
\end{enumerate}
Let $N' \leq G$ be a subgroup containing $N$ with finite index. The fact that $|N'/N|< \infty$ implies that $N'$ is coarsely equivalent to $N$. Since $N$ is coarsely embeddable into a Hilbert space $\mathcal{H}_0$, we have that $N'$ is aslo coarsely embeddable into the Hilbert space $\mathcal{H}_0$. Let $h'_{N'}: N' \to \mathcal{H}_0$ be the coarse embedding map, and let $\rho^{N'}_-,\rho^{N'}_+: [0,\infty) \to [0, \infty)$ be two non-decreasing functions with $\lim_{t \to \infty}\rho^{N'}_-(t)=\infty$, such that
$$\rho^{N'}_-(d(x, y)) \leq \|h'_{N'}(x)-h'_{N'}(y)\|\leq \rho^{N'}_+(d(x, y)),$$
for all $x,y \in N'$.

Let $S \subset G$ be a set of the representatives of the left cosets $G/N'$. Then we have a decomposition $G =\bigsqcup_{g \in S } gN'$, and the coarse embedding $h'_{N'}:N' \to \mathcal{H}_0$ can be extended to a map $h_{N'}: G \to \mathcal{H}_0$ by
$$h_{N'}(g\cdot n)=h'_{N'}(n),$$
where $g \in S$, $n \in N'$. Since every element $g' \in G$ can be uniquely written as $g'=g n$ for some $g \in S$, $n \in N'$, the extension is well-defined, and it is not a coarse embedding in general.

For any fixed element $n \in N'$, define a function $f_n: G \to \mathbb{R}$ by
$$f_{n}(g)=\|h_{N'}(g)-h_{N'}(gn)\|^2$$
for all $g \in G$.

By the coarse embeddability of $N'$, $f_{n}$ is a bounded function on $G$ for each $n \in N'$. Unfortunately, $f_n$ could be an unbounded function on $G$ if the element $n\in G$ is not in $N'$. Let $Y'_{N'}$ be the spectrum of the unital commutative $G$-invariant $C^*$-subalgebra of $\ell^{\infty}(G)$ generated by all $c_0(G)$ functions, all constant functions and all functions of the form $f_{n}$ together with their right  translations. The right action of $G$ on $\ell^{\infty}(G)$ is defined by $(\gamma f)(g)=f(g \gamma)$, for all $\gamma, g \in G$, all $f \in \ell^{\infty}(G)$.
Accordingly, the compact space $Y'_{N'}$ admits a right action induced by the restriction of the right $G$-action on $\ell^{\infty}(G)$ to the $C^*$-subalgebra which is $*$-isomorphic to $C(Y'_{N'})$.

Define a function $\phi'_{N'}: G \rtimes N' \to \mathbb{R}$ by $\phi'(g, n)=\|h_{N'}(g)-h_{N'}(gn)\|^2$, for all $g \in G$, $n \in N'$. For each $n \in N'$, the bounded function $\phi'_{N'}(\cdot, n): G \to \mathbb{R}$ extends to a continuous function on $Y_{N'}$ by the definition of $Y_{N'}$. As a result, we obtain a continuous conditionally negative definite function $\phi'_{N'}: Y'_{N'} \rtimes N' \to \mathbb{R}$ by extending the map $\phi'_{N'}: G \rtimes N' \to \mathbb{R}$.
Then, replace $Y'_{N'}$ with the space of all the probability measures on $Y'_{N'}$, denoted by $Y_{N'}$. The space $Y_{N'}$ is a second countable compact space equipped with the weak-$*$ topology (c.f. \cite{DouglasBook}). The $G$-action on $Y'_{N'}$ induces a right action on $Y_{N'}$. We define a continuous function $\phi_{N'}: Y_{N'} \rtimes N \to \mathbb{R}$ by
$$\phi_{N'}(m, g)=\int_{Y'_{N'}} \phi'(y, g) d(mg),$$
for all $m \in Y_{N'}$, $g \in N'$.

\begin{prop}
The countinuous function $\phi_{N'}: Y_{N'} \rtimes N' \to \mathbb{R}$ is a proper conditionally negative definite function.
\end{prop}

\begin{proof}
By Proposition \ref{prop_proper_CND_on_groupoid}, we have that $\phi_{N'}$ is a continuous conditionally negative definite functionon the transformation groupoid $Y_{N'}\rtimes N'$.

By the definition of the map $\phi_{N'}: G \times N' \to \mathbb{R}$, we have $\rho^{N'}_-(|n|)\leq \phi_{N'} (g, n)$ for all $g \in G$, $n \in N'$. It follows that the extension $\phi_{N'}: Y_{N'}\rtimes N' \to \mathbb{R}$ satisfies the inequality $\rho^{N'}_-(|n|)\leq \phi'(y, n)$ for all $y \in Y'_{N'}$, $n \in N'$. It follows that the conditionally negative definite function $\phi_{N'}: Y_{N'}\rtimes N' \to \mathbb{R}$ is proper.
\end{proof}

\begin{rem}
For each subgroup $N'\leq G$, we can find a compact space $Y_{N'}$ with a right $G$-action, such that the groupoid $Y_{N'} \rtimes N'$ is a-T-menable in the sense that $Y_{N'}\rtimes N'$ admits a proper, continuous conditionally negative definite function.
\end{rem}

\begin{lem}
  For each finite subgroup $G_0 \leq G$, the compact space $Y_{N'}$ is $G_0$-contractible.
\end{lem}
\begin{proof}
Since $Y_{N'}$ is a convex set, it contracts to a point $y_0 \in Y_{N'}$. For any finite subgroup $G_0 \leq G$, it is obvious that $\sum_{g \in G_0} \frac{1}{|G_0|} y_0g \in Y_{N'}$. It follows that $Y_{N'}$ is $G_0$-contractible to the point $\sum_{g \in G_0} \frac{1}{|G_0|} y_0g$.
\end{proof}

Let $\mathcal{F}$ be the set of all subgroups of $G$ containing $N$ with finite index. For each $N' \in \mathcal{F}$, we can find a compact space $Y_{N'}$. We then define a compact topological space
$$
Y:= \prod_{N'\in \mathcal{F}} Y_{N'}.
$$
The topology on $Y$ is the product topology, and the $G$-action is then defined by
$$\displaystyle g\cdot \left(y_{\tiny N'}\right)_{N' \in \mathcal{F}}=\left( y_{\text{\tiny $N'$}} g^{-1}\right)_{N' \in \mathcal{F}}$$
for all $g \in G$, $\left(y_{N'}\right)_{N' \in \mathcal{F}}\in Y$.

\begin{prop}\label{prop_Haag_subgrpd}
For each $N'_0 \in \mathcal{F}$, the associated transformation groupoid $Y\rtimes N'_0$ is a-T-menable.
\end{prop}

\begin{proof}
Define a continuous conditionally negative definite function on the transformation groupoid $Y \rtimes N'_0$ by
$$\phi \left(\left(y_{N'}\right)_{N' \in \mathcal{F}}, n\right)=\phi_{\text{\tiny $N'_0$}}(y_{\text{\tiny $N'_0$}}, n),
$$
for all $n \in N'_0$, $\left(y_{N'}\right)_{N' \in \mathcal{F}} \in Y$.
It is easy to check that this map is a proper conditionally negative definite function.
\end{proof}
\begin{rem}
  The domain of the conditionally negative definite function is $Y \rtimes N_0'$ instead of $Y \rtimes G$ for each $N_0' \in \mathcal{F}$.
\end{rem}
Assume the quotient group $G/N$ coarsely embeds into a Hilbert space. In Section 3, we obtained a compact metrizable space $X$ such that $X \rtimes G/N$ is a-T-menable, and a proper $G/N$-$C^*$-algebra, denoted by $\mathcal{A}(X, \mathcal{H})$. In the rest of this section, we will formulate the twisted Baum--Connes conjecture for $G$ with coefficients in $C(Y)\widehat{\otimes} \mathcal{A}(X, \mathcal{H})\widehat{\otimes}B$.

For each $d>0$, let $P_d(G)$ be the Rips complex endowed with the spherical metric as defined in Section 2. Take a countable dense subset $Z_d \subset P_d(G)$, such that $Z_d \subset Z_d'$ whenever $d < d'$.
 Let $\big(\mathcal{H}_x\big)_{x \in X}$ be the continuous field of Hilbert spaces such that the transformation groupoid $X \rtimes G/N$ acts properly on $\big(\mathcal{H}_x\big)_{x \in X}$ by affine isometries. For every $(x, g) \in X \rtimes G/N$, there is an affine isometry $V_{(x, g)}: \mathcal{H}_{xg} \to \mathcal{H}_x$ and a continuous section $b: X \times G/N \to \mathcal{H}$ with $b(x, g) \in \mathcal{H}_x$, such that for every $v \in \mathcal{H}_{(x, g)}$, $V_{(x, g)}(v)=U_{(x, g)}(v)+b(x, g)$, for all $(x, g) \in X \rtimes G/N$, where  $U_{(x, g)}: \mathcal{H}_{xg}\to \mathcal{H}_x$ is a linear isometry for each $(x, g) \in X \rtimes G/N$. The map $b: G/N \to \Theta(X, \mathcal{H})$ is called the cocycle associated with the groupoid action of $X \rtimes G/N$ on the continuous field of Hilbert spaces $\big(\mathcal{H}_x\big)_{x \in X}$. By the construction of the continuous field of Hilbert spaces, $b(x, e)=0 \in \mathcal{H}_x$, for all $x \in X$. By coarse embeddability and the definition of $b$, we have that $\inf_{x \in X}\|b(x, g)\|_{\mathcal{H}_x} \to \infty$ as $|g| \to \infty$.

Let $\Theta(X, \mathcal{H})$ be the space of all continuous sections associated with the continuous field of Hilbert spaces $\big(\mathcal{H}_x\big)_{x \in X}$.
We will define a second countable, locally compact topological space $W$, and a proper $G/N$-action on $W$.
As a set, denote {\bf $W=\mathbb{R}_+ \times \bigsqcup_{x \in X} \mathcal{H}_x $}, where $\mathbb{R}_+$ is the set of all non-negative numbers. A topology can be defined as follows. Let $\left\{(t_i, x_i, v_i)\right\}_i$ be a net in $W$, it converges to a point $(t, x, v) \in W$ if
\begin{enumerate}
\item $x_i \to x$, and $t_i^2+\|v_i\|_{\mathcal{H}_x}^2 \to t^2+\|v\|^2_{\mathcal{H}_x}$;

\item for any continuous section $e:X \to \mathcal{H}$, we have $\langle e(x_i), v_i \rangle_{\mathcal{H}_{x_i}}\to \langle e(x), v\rangle_{\mathcal{H}_x}$, as $i \to \infty$.
\end{enumerate}

The topology on $W$ can also be characterized in terms of its base consisting of the following open sets.
For each point $(t_0, x_0, v_0) \in \mathbb{R}_+ \times \mathcal{H}$, each $U_{x_0} \subset Y$ a neighborhood of $x_0$, each constant $\epsilon>0$, each section $s \in \Theta(X, \mathcal{H})$, define open sets

 $$\left\{(t, x, v): |(t^2-t_0^2)+(\|v\|^2_{\mathcal{H}_x}-\|v_0\|^2_{\mathcal{H}_{x_0}})|< \epsilon, x \in U_{x_0}\right\},$$
 and
$$\left\{(t, x, v): |\langle s(x), v-v_0\rangle_{\mathcal{H}_x}| <\epsilon, x \in U_{x_0} \right\}.$$
The topology on $W$ is generated by sets of the above forms. The space $W$ is a second countable, locally compact and Hausdorff space. By the construction of the space $X$, it is obvious that $X$ is second countable and separable. We obtain a countable basis for the topology on $W$ by taking $\epsilon$ and $t_0$ in the rational numbers $\mathbb{Q}$, $x_0$ in a countable dense subset of $X$, and $v_0$ in countable dense subset in each fiber $\mathcal{H}_x$ in the definition of the above open subsets. As a consequence, the space $W$ is second countable. For local compactness, for each $R>0$, the subset
$$\left\{(t, x, v)\in W: t^2+\|v\|^2_{\mathcal{H}_x}\leq R^2\right\}$$
 is compact. To see this, we first choose a net
 $$\{(t_i, x_i, v_i)\}_{i \in I} \in\left\{(t, x, v)\subset W: t^2+\|v\|^2_{\mathcal{H}_x}\leq R^2\right\}.$$
 Since $\{t_i\}_i$ is bounded, a convergent subnet exists. Without loss of generality, we assume $\{t_i\}_i$ converges to $t_0$. We can also assume $\{x_i\}_i$ converges to $x_0$ due to the compactness of the space $X$. Fix an orthonormal basis $\{e_n\}_{n\geq 0}$ for the fiber $\mathcal{H}_{x_0}$, one can find a sequence of continuous sections $\{e_n \in \Theta(X, \mathcal{H})\}$ such that $e_n(x_0)=e_n$, for all $n \geq 0$. By the diagonal argument, we can find a subnet $\{(t_i, x_i, v_i)\}_i$ such that, $\lim_{i} \langle e_n(x_i), v_i\rangle_{\mathcal{H}_{x_i}}$ exists for each $n \geq 0$. Let $\lambda^i_n=\langle e_n(x_i), v_i\rangle_{\mathcal{H}_{x_i}}$, and $\lambda_n= \lim_{i} \lambda^i_n$. It is easy to check that $\{(t_i, x_i, v_i)\}_i$ converges to the point $(t_0, x_0, v_0)$ under the above topology, where $v_0=\sum_{n=1}^{\infty} \lambda_n e_n$.

Obviously, $\left\{(t, x, v) \in \mathbb{R}_+ \times \mathcal{H}: |t|^2+\|v-s(x)\|_{\mathcal{H}_x}^2< R, x \in U_{x_0}\right\}$ is an open subset of the space $W=\mathbb{R}_+ \times \mathcal{H}$, where $R$ is a positive constant, $U_{x_0}$ is an open neighborhood of $x_0$ in $X$, and $s: U_{x_0} \to \mathcal{H}$ is a local continuous section over $U_{x_0}$.

Now, let us define the $G/N$-action on $W$. Let $ g \in G/N$, $(t, x, v) \in W$, we define
$$g \cdot (t, x, v)=(t, xg^{-1}, V_{(xg^{-1}, g)}(v)).$$
For every $g \in G/N$, every continuous section $s \in \Theta(X, \mathcal{H})$, it follows that $g \cdot s$ is also a continuous section. In addition, every function on $W$ of the form $(t, x, v)\mapsto t^2+ \|v-v_0\|_{\mathcal{H}_x}^2$ is continuous under the topology of $W$, where $x \in X$, $v, v_0 \in \mathcal{H}_{x}$. As a result, the action of $G/N$ on $W$ is well-defined.

 According to \cite{BaCoHig}, the properness of the $G/N$-action on $W$ is equivalent to the fact that the set $\left\{g \in G/N: g\cdot K \cap K \neq \emptyset\right\}$ is finite for each compact subset $K \subset W$.
\begin{prop}\label{prop_Action_on_W}
The action $G/N \curvearrowright W$ is proper.
\end{prop}
\begin{proof}
For each positive integer $n>0$, let $K_n=\left\{(t, x, v) \in W: t^2+\|v\|^2_{\mathcal{H}_x}\leq n \right\}$. Since $K_n$ is compact for each $n>0$, and $W=\bigcup_{n>0} K_n$, it suffices to show that the set $\left\{g \in G/N: g\cdot K_n \cap K_n \neq \emptyset\right\}$ is finite for each $n>0$. Let $(t, x, v) \in K_n$, $g \cdot (t, x, v)=(t, xg^{-1}, V_{(xg^{-1}, g)}(v))$ for all $g \in G/N$. Since $V_{(xg^{-1}, g)}(v)=U_{(xg^{-1}, g)}(v)+ b(xg^{-1}, g)$ and $\inf_{x \in X}\|b(x, g)\|_{\mathcal{H}_x} \to \infty$ as $|g| \to \infty$, there exists some $R>0$, such that
$$\inf_{x \in X}\|b(xg^{-1}, g)\|^2>n+\sup_{x \in X}\|v\|^2_{\mathcal{H}_{x}}$$
for all $|g|>R$. Since $U_{(xg^{-1}, g)}$ is an isometry for each $(xg^{-1}, g) \in X \rtimes G/N$, so $g\cdot (t, x, v) \notin K_n$. By the properness of the metric on $G/N$, the set $\left\{g \in G/N: g\cdot K_n \cap K_n \neq \emptyset\right\}$ is finite for each $n>0$. Thus, the $G/N$-action on $W$ is proper.
\end{proof}

Note that the $C^*$-algebra $C_0(W)$ of all continuous functions on $W$ vanishing at infinity is contained in the center of the $C^*$-algebra $\mathcal{A}(X, \mathcal{H})$, see \cite{Tu_BC} for more details. For each open subset $U \subset  X$ and a continuous affine distribution $\left\{E_a(y)\right\}_{y \in U}$ over $U$, we have that
$$\mathcal{A}(U, E_a)=C_0(W_U)\cdot \mathcal{A}(U, E_a)$$
where $E_a=\bigsqcup_{y \in U} E_a(y)$ and $W_{(U, E_a)}:=\left\{(t, y, v): t \in \mathbb{R}_+, y \in U, v \in E^0_a(y)\right\} \subset W$. If $U\subset V$ are open subset of $X$, and $\left\{E_a(y)\right\}_{y \in U}$ and $\left\{E_b(y)\right\}_{y \in V}$ are continuous affine distributions over $U$ and $V$ respectively, with $E_a(y) \subset E_b(y)$ for each $y \in U$, then the fiber-wise defined Bott map $\beta_{U, ba}$ takes $C(W_{(U, E_a)})$ into $C_0(W_{(V, E_b)})$. Accordingly, the $C^*$-algebra $C_0(W)$ can be viewed as a direct limit $\varinjlim C_0(W_{U, E_a})$. As a result, $C_0(W)\cdot \mathcal{A}(X, \mathcal{H})$ is dense in $\mathcal{A}(X, \mathcal{H})$. We have an action of $G/N$ on the $C^*$-algebra $\mathcal{A}(X, \mathcal{H})$ as defined in Section 3. The $C^*$-algebra $C_0(W)$ is contained in the center of the $C^*$-algebra, and the properness of the $G/N$-action on $W$ implies that the action of $G/N$ on $\mathcal{A}(X, \mathcal{H})$ is proper. In the rest of this section, we lift $G/N$-actions on $W$ and $\mathcal{A}(X, \mathcal{H})$ to $G$-actions via the quotient map $G \to G/N$, and we use the same notations for $G$ and $G/N$ actions.

\subsection{Twisted Roe algebras and twisted localization algebras} In the rest of this section, we will define the twisted version of Roe algebras and localization algebras. Let $d>0$, and let $P_d(G)$ be the simplical complex at scale $d$, endowed with the simplicial metric. Let $Z_d$ be the countable dense $G$-invariant subset of $P_d(G)$ consisting of all linear combinations $\sum_{g \in G} c_g g$ with $c_g \in \mathbb{Q}$ and $c_g\neq c_g'$ for any pair $g\neq g'$. Note that for each subcomplex $C \subset P_d(G)$, the intersection $C \cap Z_d$ is non-empty and $C=\widebar{C \cap Z_d}$.

Since the left translation action $G \curvearrowright P_d(G)$ is proper and cocompact, one can define a coarse $G$-equivariant map $J:P_d(G) \rightarrow G$. This map is defined as follows. For each $d>0$, we can fix a bounded subset $\Delta_d \subset Z_d $ such that
 \begin{enumerate}
 \item $G\cdot \Delta_d= Z_d$,
 \item for every $z \in Z_d$, there exist unique $x \in \Delta_d$ and unique $g \in G$, such that $z=g\cdot x$,
 \item for $d<d'$, $\Delta_d \subset \Delta_{d'}$.
 \end{enumerate}

By condition (2), for every element $z \in Z_d$, there exists a unique element $g \in G$, such that $g^{-1}z \in \Delta_d$. The map $J:Z_d(G) \to G$ can be defined as $J(z)=g$, where $z \in Z_d$, and $z= gx$, for some $g \in G$ and $x \in \Delta_d$. Note that the map $J$ is $G$-equivariant. Indeed, for all $\gamma \in G$, $z \in Z_d$, we have $J(z)^{-1}z \in \Delta_d$, so $(\gamma\cdot J(z))^{-1}(\gamma z)=J(z)^{-1}z \in \Delta_d$. Hence $J(\gamma z)=\gamma J(z)$.

Define an open set
{\bf $$
O_R(g)=\left\{ (t, x, v): t^2+\|v - b(x, g)\|_{\mathcal{H}_x}^2 < R^2\right\} \subset W.
$$}
for all $g \in G$.
We need the following lemma to define twisted Roe algebras.
\begin{lem}\label{lem_O_R_invar}
If $z \in Z_d$, $\gamma \in G$, then $\gamma\cdot O_R(J(z))=O_R(J(\gamma z))$.
\end{lem}
\begin{proof}
For any $(t, x, v) \in O_R(z)$, the action is given by $\gamma \cdot(t, x, v)=(t, x \gamma^{-1}, V_{(x\gamma^{-1}, \gamma)} (v))$, and $b(x\gamma^{-1}, J(\gamma z))=b((x\gamma^{-1}, \gamma)(x, J(z)))=V_{(x\gamma^{-1}, \gamma)}b(x, J(z))$.
Since $V_{(x\gamma^{-1}, \gamma)}$ is an isometry for each $x \in X$, we have that $t^2+\|V_{(x\gamma^{-1}, \gamma)} ( v)-b(x\gamma^{-1}, J(\gamma z))\|^2 < R^2$, for all $\gamma \in G$.
As a result, $\gamma \cdot (t, x, v) \in O_R(\gamma z)$.
\end{proof}

Let $H_0$ be any fixed separable complex Hilbert space with infinite dimension. Let $K_G$ be the algebra of all compact operators on $H_0\widehat{\otimes} \ell^2(G)$, and $B$ a $G$-$C^*$-algebra.  Define a $G$-action on $H_0 \widehat{\otimes}\ell^2(G)$ by first defining $\gamma(v \widehat{\otimes}\delta_g)=v\widehat{\otimes}\delta_{\gamma g}$, for all $\gamma, g \in G$ and $v \in H_0$, and then extending linearly to $H_0\widehat{\otimes} \ell^2(G)$. The algebra $B \widehat{\otimes}K_G$ is equipped with a unitary $G$-action by $\gamma \cdot (b \widehat{\otimes}T)= \gamma \cdot b  \widehat{\otimes}\gamma T \gamma ^*$, for all $b \in B$, $T \in K_G$, $\gamma \in G$.

\begin{defn}
For an element $S \in C(Y)\widehat{\otimes} \mathcal{A}(X, \mathcal{H}) \widehat{\otimes}B \widehat{\otimes}K_G$, we can define the {\bf support} of $S$, denoted by $\text{supp}(S)$, to be the complement of the set of $(t, x, v) \in \mathbb{R}_+ \times \mathcal{H}$ such that there exists $f \in C_0(\mathbb{R}_+ \times \mathcal{H})$ with $f(t, x, v) \neq 0$,  $(1_Y\widehat{\otimes} f \widehat{\otimes}k)\cdot S=0$, for all $k \in B \widehat{\otimes}K_G$.
\end{defn}

\begin{defn}
Define \textbf{the algebraic twisted Roe algebra}, denoted by $$C^{\ast}_{alg}(P_d(G), C(Y)\widehat{\otimes}\mathcal{A}(X, \mathcal{H})\widehat{\otimes}B)^G,$$
as the set of all the functions $T$ on $ Z_d \times Z_d$ such that
\begin{enumerate}
\item $T(y, z)\in C(Y) \widehat{\otimes}\mathcal{A}(X, \mathcal{H}) \widehat{\otimes} B \widehat{\otimes}K_G$,
\item there exists $M>0$ such that $\|T(y, z)\|\leq M$, for all $x,y \in Z_d$,
\item there exists $L>0$, such that
\begin{center}
$\# \{y:T(y, z)\neq0\}<L$ and $\#\{z:T(y, z)\neq 0\}<L$,
\end{center}

\item there exists $r_1 \geq 0$, such that $T(y, z)=0$, for any $y, z \in Z_d$ with $d(y, z)>r_1$,
\item  there exists $r_2>0$, such that $supp(T_{y, z})= O_{r_2}(J(y))$,

\item the operator $T$ is $G$-invariant, i.e., $\gamma(T_{\gamma^{-1}y, \gamma^{-1}z})=T_{y, z}$, for all $\gamma \in G$, $y, z \in Z_d$.
\end{enumerate}
The algebraic twisted Roe algebra is equipped with a $\ast$-algebra structure by the matrix operations.
\end{defn}

\begin{rem}
By Lemma \ref{lem_O_R_invar}, the above condition (4) makes sense.
\end{rem}
Let
$$E=\left\{\sum_{z \in Z_r} a_z [z]: a_z \in C(Y) \widehat{\otimes} \mathcal{A}(X, \mathcal{H})\widehat{\otimes} B \widehat{\otimes}K_G, \sum_{z \in Z_r}a_z^*a_z~\text{converges in norm}
\right\}.$$
It is a $G$-Hilbert module over $C(Y)\widehat{\otimes} \mathcal{A}(X, \mathcal{H})\widehat{\otimes} B \widehat{\otimes}K_G$.
For all $\sum_{z \in Z_r}a_z[z], \sum_{z \in Z_r}b_z[z] \in E$, $a \in C(Y)\widehat{\otimes} \mathcal{A}(X, \mathcal{H})\widehat{\otimes} K_G$
$$
\left\langle\sum_{z \in Z_r}a_z[z],\sum_{z \in Z_r}b_z[z]\right\rangle:=\sum_{z \in Z_r}a_z^*b_z,
$$
and
$$\left(\sum_{z \in Z_r}a_z[z]\right)a:=\sum_{z \in Z_r}a_z a[z].
$$
The action of $G$ is given by
$$g \cdot \left(\sum_{z \in Z_r} a_z [z]\right):=\sum_{z \in Z_r} (g\cdot a_z) [gz].$$
Define a $\ast$-representation of $C^{\ast}_{alg}(P_d(G), C(Y)\widehat{\otimes}\mathcal{A}(X, \mathcal{H}) \widehat{\otimes}B)^G$ on $E$ by
$$T\left(\sum_{z\in Z_r}a_z [z]\right):=\sum_{y \in Z_r}\left(\sum_{z\in Z_r}T_{y, z}a_z\right)[y]$$

According to the definition of the algebraic twisted Roe algebra, the $\ast$-representation on $E$ is well-defined. \textbf{The twisted Roe algebra}, denoted by {\bf $$C^{\ast}(P_d(G), C(Y) \widehat{\otimes} \mathcal{A}(X, \mathcal{H})\widehat{\otimes}B)^G,$$} is defined to be the completion of the algebraic twisted Roe algebra under the operator norm in $B(E)$, where $B(E)$ is the $C^*$-algebra of all adjointable module homomorphisms.

Let $C_{L,alg}^{\ast}(P_d(G), C(Y) \widehat{\otimes}\mathcal{A}(X, \mathcal{H})\widehat{\otimes}B)^G$ be the set of all bounded, uniformly norm-continuous functions
$$g:\mathbb{R}_+ \rightarrow C_{alg}^{\ast}(P_d(G), C(Y)\widehat{\otimes} \mathcal{A}(X, \mathcal{H})\widehat{\otimes}B)^G,$$
such that
\begin{center}
$\text{propagation}(g(t))\to 0$, as $t \to \infty$.
\end{center}
Taking the completion with respect to the norm
$$\|g\|=\sup_{t \in \mathbb{R}_+} \|g(t)\|,$$
we have \textbf{the twisted localization algebra}, denoted by $C^{\ast}_L(P_d(G), C(Y)\widehat{\otimes}\mathcal{A}(X, \mathcal{H})\widehat{\otimes}B)^G$.
\begin{rem}
  By Proposition \ref{Prop_Roe_coarse_invar} and Remark \ref{rem_Roe_independent_on_admissible_system}, we have that the twisted Roe algebras and the twisted localization algebras are independent of the choice of the countable dense subset $Z_d$ for each $d>0$. Indeed, if we have two countable dense $G$-invariant subset $Z_d, Z'_d \subset P_d(G)$, then $Z_d \cup Z'_d$ is also a countable dense $G$-invariant subset of $P_d(G)$. Following the constructions in Proposition \ref{Prop_Roe_coarse_invar} and Remark \ref{rem_Roe_independent_on_admissible_system}, the inclusion map $Z_d \to Z_d \cup Z'_d$ induces an isomorphisms between the twisted Roe algebras defined by choosing $Z_d$ and $Z_d \cup Z'_d$, while the inclusion map $Z'_d \to Z_d \cup Z'_d$ induces an isomorphism between the twisted Roe algebras defined by choosing $Z'_d$ and $Z_d \cup Z'_d$. As a result, the definitions of twisted Roe algebras and twisted localization algebras (by the similar argument) are independent of the choice of the countable dense $G$-invariant subset $Z_d$.
\end{rem}
There is a natural \textbf{evaluation-at-zero map}
$$ev:C_L^{\ast}(P_d(G), C(Y)\widehat{\otimes}\mathcal{A}(X, \mathcal{H})\widehat{\otimes}B)^G \rightarrow
C^{\ast}(P_d(G), C(Y) \widehat{\otimes}\mathcal{A}(X, \mathcal{H})\widehat{\otimes}B)^G,$$
defined by
$ev(g)=g(0)$, for all $g \in C_L^{\ast}(P_d(G), C(Y)\widehat{\otimes}\mathcal{A}(X, \mathcal{H})\widehat{\otimes}B)^G$. Obviously, it is a $*$-homomorphism. The evaluation-at-zero map induces a homomorphism on $K$-theory
$$
ev_*:K_*(C_L^{\ast}(P_d(G), C(Y)\widehat{\otimes}\mathcal{A}(X, \mathcal{H})\widehat{\otimes}B)^G) \rightarrow
K_*(C^{\ast}(P_d(G), C(Y) \widehat{\otimes}\mathcal{A}(X, \mathcal{H})\widehat{\otimes}B)^G).
$$

\section{The $K$-theory of twisted Roe algebras and twisted localization algebras}

In this section, we will prove the following twisted Baum--Connes conjecture for groups which are extensions of coarsely embeddable groups.

\begin{thm}\label{thm_twist_BC}
Let $1 \rightarrow N \rightarrow G\rightarrow G/N \rightarrow 1 $ be a short exact sequence of countable discrete groups. Assume $N$ and $G/N$ can be coarsely embedded into Hilbert spaces. The map
\begin{equation*}
ev_{\ast}: \displaystyle \lim_{d \to \infty}K_{\ast}(C_{L}^{\ast}(P_d(G),C(Y)\widehat{\otimes} \mathcal{A}(X, \mathcal{H})\widehat{\otimes}B)^G)
\rightarrow \;\;\;\;\;\;\;\;\;\;\;\;\;\;\;\;\;\;\;\;\;\;\;\;\;\;\;\;\;\;\;\;\;\;\;\;\;\;\;\;\;\;
\end{equation*}
\begin{equation*}
 \displaystyle \lim_{d \to \infty}K_{\ast}(C^{\ast}(P_d(G), C(Y)\widehat{\otimes}\mathcal{A}(X, \mathcal{H})\widehat{\otimes}B)^G)
\end{equation*}
induced by the evaluation-at-zero map is an isomorphism.
\end{thm}

 To prove this theorem, the main idea is to decompose the twisted Roe algebra into ideals whose $K$-theory can be easily computed, and then use the Mayer--Vietoris sequence and the Five Lemma to piece them all together. The decomposition of the twisted Roe algebra relies on the structure of the $G$-action on the space $W$ (Proposition \ref{prop_Action_on_W}). Let us recall an equivalent definition of proper action of a discrete group on a topological space (Definition 1.3 in \cite{BaCoHig}).

\begin{defn}
Let $\Gamma$ be a countable discrete group, $X$ a topological space equipped with a continuous action of $\Gamma$. The $\Gamma$-action is called proper if

\begin{enumerate}
\item $X$ is second countable,

\item $X/ \Gamma$ is second countable,

\item for every $x \in X$ there is a $\Gamma$-invariant neighborhood $U$ of $x$ and a finite subgroup $\Gamma_0$ of $\Gamma$, such that there exists a continuous $\Gamma$-map $U \to \Gamma/ \Gamma_0$.
\end{enumerate}
\end{defn}

In Section 4, we obtained a proper $G/N$-action on a topological space $W$ (Proposition \ref{prop_Action_on_W}) and lifted the $G/N$-action to a $G$-action on $W$ via the quotient map $G \to G/N$. For any $R>0$, let $O_R(g)=\left\{ (t, x, v): t^2+\|v - b(x, g)\|_{\mathcal{H}_x}^2 < R^2\right\} \subset W$. We have checked that $O_R=\bigcup_{g \in G/N} O_R(g)$ is $G/N$-invariant as is the closure of $O_R$  by Lemma \ref{lem_O_R_invar}. The restriction of the $G/N$-action on the closure $\widebar{O_R}$ is proper and cocompact, because $\widebar{O_R}=G/N \cdot \widebar{O_R(e)}$. Fix $R_0>0$, for any $R <R_0$, every point in $O_R$ has a $G/N$-invariant neighborhood identical to $V\times_{N'/N} G/N$, where $V$ is an open subset of $O_{R_0}(e)$, and $N'$ is a subgroup of $G$ containing $N$ with finite index. Throughout this paper, we identify sets of the form $V \times_{N'} G/N$ with $G/N \cdot V$ via $(v, g) \mapsto vg$, and $V \times \{g\}$ is identified with $g \cdot V$ for all $ g \in G/N$.
 Those open sets together with the difference $\widebar{O_{R_0}(e)}-\widebar{O_R(e)}$ comprise an open cover of $\widebar{O_{R_0}(e)}$. By compactness of $\widebar{O_{R_0}(e)}$, we can find a finite open cover $ \left\{ V_i \times_{N_i/N} G/N\right\}^k_{i=1}$ for $O_R$, where $V_i \subset \{(x, s, v): s^2+\|v\|_{H_x}^2< R^2\}$ and $V_i$ is $N_i$-invariant, for some $N_i$ a subgroup of $G$ containing $N$ with finite index. We will consider the restriction of the $G$-action on $W$ to an $N_i$-action on $V_i$. Let $G \cdot V_i$ denote the union $\bigsqcup_{g \in G/{N_i}} g \cdot V_i$.

\begin{defn}\leavevmode
\begin{enumerate}
\item For any element $T \in C^*(P_d(G), C(Y)\widehat{\otimes}\mathcal{A}(X, \mathcal{H})\widehat{\otimes}B)^G$,  the support of $T$ is defined to be the set
$$\text{supp}(T)=\left\{ (y, z, v) \in Z_d \times Z_d \times (\mathbb{R}_+\times \mathcal{H}): v \in \text{supp}(T_{y, z})\right\}.$$
\item For any element $a \in C_L^*(P_d(G), C(Y)\widehat{\otimes}\mathcal{A}(X, \mathcal{H})\widehat{\otimes}B)^G$, the support of $a$ is defined to be
$$\text{supp}(a)=\bigcup_{t\geq 0} \text{supp}(a(t)).$$
\end{enumerate}
\end{defn}

For each open subset $U \subset W=\mathbb{R}_+ \times \mathcal{H}$, define $C^{\ast}(P_d(G),C(Y) \widehat{\otimes}\mathcal{A}(X, \mathcal{H})\widehat{\otimes} B)_{U}^G$ to be the $C^*$-subalgebra of $C^{\ast}(P_d(G),C(Y) \widehat{\otimes}\mathcal{A}(X, \mathcal{H})\widehat{\otimes}B)^G$ generated by all operators with support contained in $Z_d \times Z_d \times U$.

We can also define the localization algebra, denoted by $C_L^{\ast}(P_d(G),C(Y)\widehat{\otimes}\mathcal{A}(X, \mathcal{H}\widehat{\otimes}B)_{U}^G$, to be the $C^*$-subalgebra of $C_L^{\ast}(P_d(G),C(Y)\widehat{\otimes}\mathcal{A}(X, \mathcal{H})\widehat{\otimes}B)^G$ generated by all paths $a(t)$ with support contained in $Z_d \times Z_d \times U$. We have an evaluation-at-zero map
$$ ev: C_L^{\ast}(P_d(G),C(Y)\widehat{\otimes}\mathcal{A}(X, \mathcal{H})\widehat{\otimes}B)_U^G \rightarrow C^*(P_d(G),C(Y) \widehat{\otimes}\mathcal{A}(X, \mathcal{H})\widehat{\otimes}B)_U^G$$
given by $ev(f)=f(0)$, where $f \in C_L^{\ast}(P_d(G),C(Y)\widehat{\otimes}\mathcal{A}(X, \mathcal{H})\widehat{\otimes}B)_U^G$.

\begin{prop}
Let $G \cdot V_i=\bigcup_{g \in G/{N_i}} g \cdot V_i$. The homomorphism,
\begin{equation*}
 ev_{\ast}: \displaystyle \lim_{d \to \infty}K_{\ast}(C_L^{\ast}(P_d(G),C(Y)\widehat{\otimes}\mathcal{A}(X, \mathcal{H})\widehat{\otimes}B)_{G \cdot V_i}^G) \rightarrow \;\;\;\;\;\;\;\;\;\;\;\;\;\;\;\;\;\;\;\;\;\;\;\;\;\;\;\;\;\;\;\;\;\;\;\;\;\;\;\;\;\;\;\;\;\;\;\;\;\;\;\;\;\;\;\;\;\;\;\;
  \end{equation*}
  \begin{equation*}
  \;\;\;\;\;\;\;\;\;\;\;\;\;\;\;\;\;\;\;\;\;\;\;\;\;\;\;\;\;\;\;\;\;\;\;\;\;\;\;\;\;\;\;\;\;\;\;\;\;\;\;\;\;\;\;\;\;\;\;\; \displaystyle \lim_{d \to \infty}K_{\ast}(C^*(P_d(G),C(Y) \widehat{\otimes}\mathcal{A}(X, \mathcal{H}\widehat{\otimes}B))_{G \cdot V_i}^G),
  \end{equation*}
induced by the evaluation-at-zero map  on $K$-theory is an isomorphism.
\end{prop}

In order to prove this theorem, we need some lemmas.
For each $N_i$ and $d>0$, we define the Roe algebra with coefficients in
$C(Y) \widehat{\otimes}\mathcal{A}(X, \mathcal{H})_{V_i}\widehat{\otimes}B$ as follows. Since $N_i \leq G$, we can assume $P_d(N_i) \subset P_d(G)$ has the restricted spherical metric of $P_d(G)$. Let $Z^i_d=Z_i \cap P_d(N_i)$. By the definition of $Z_d$, the set $Z^1_d$ is non-empty for all $i$.
The algebraic Roe algebra, denoted by
$$C_{alg}^*(P_d(N_i), C(Y)\widehat{\otimes} \mathcal{A}(X, \mathcal{H})_{V_i}\widehat{\otimes}B)^{N_i},$$
is defined to be the set of all matrices $\big( T_{y, z} \big)_{y, z \in Z^i_d}$
which represent bounded, finite propagation, locally compact, and $N_i$-invariant operators on the Hilbert module $$\ell^2(Z^i_d) \widehat{\otimes} \mathcal{A}(X, \mathcal{H})_{V_i} \widehat{\otimes}B\widehat{\otimes} K_G,$$
where $K_G$ is the algebra of all the compact operators on $\ell^2(G) \widehat{\otimes} H $ endowed with the tensor product unitary representation of $N_i$. Define the Roe algebra
$$C^*(P_d(N_i), C(Y)\widehat{\otimes}\mathcal{A}(X, \mathcal{H})_{V_i})^{N_i}$$
to be the norm closure of the algebraic Roe algebra  $C_{alg}^*(P_d(N_i), C(Y)\widehat{\otimes}\mathcal{A}(X, \mathcal{H})_{V_i}\widehat{\otimes}B)^{N_i}$, on the Hilbert module
$\ell^2(Z^i_d) \widehat{\otimes} \mathcal{A}(X, \mathcal{H})_{V_i} \widehat{\otimes}B\widehat{\otimes} K_G$.

For each $i$, we can define an inclusion homomorphism
$$\imath_i: C^{\ast}(P_d(N_i), C(Y)\widehat{\otimes} \mathcal{A}(X, \mathcal{H})_{V_i}\widehat{\otimes}B)^{N_i} \rightarrow C^{\ast}(P_d(G), C(Y)\widehat{\otimes}\mathcal{A}(X, \mathcal{H})\widehat{\otimes}B)_{G\cdot V_i}^G,$$
 between the $C^{\ast}$-algebras by
\[ \imath_i(T)_{x,y} =
 \begin{cases*}
    g^{-1}(T_{gx,gy}) \quad& if $\exists $ $g \in G  $  such that $gx, gy \in P_d(N_i)$,\\
    0 & otherwise.
\end{cases*}\]
Similarly, we can define the inclusion map on the localized version,
$$\imath_{L, i}: C_L^{\ast}(P_d(N_i),C(Y) \widehat{\otimes}\mathcal{A}(X, \mathcal{H})_{V_i}\widehat{\otimes}B)^{N_i} \rightarrow C_L^{\ast}(P_d(G), C(Y) \widehat{\otimes}\mathcal{A}(X, \mathcal{H})\widehat{\otimes}B)_{G \cdot V_i}^G.$$

The following result is crucial to reducing twisted Roe algebras and twisted localization algebras of $G$ to twisted Roe algebras and twisted localization algebras associated with its subgroups $N_i$.

\begin{lem}\label{lem_reduce_to_subgroup}
For each $V_i$, the maps
$$(\imath_i)_{\ast}:K_{\ast}(C^{\ast}(P_d(N_i), C(Y)\widehat{\otimes} \mathcal{A}(X, \mathcal{H})_{V_i}\widehat{\otimes}B)^{N_i}) \rightarrow K_{\ast}(C^{\ast}(P_d(G), C(Y) \widehat{\otimes}\mathcal{A}(X, \mathcal{H})\widehat{\otimes}B)_{G \cdot V_i}^G),$$

and
\begin{equation*}
(\imath_{L, i})_{\ast}:\displaystyle \lim_{d \to \infty}K_{\ast}(C_L^{\ast}(P_d(N_i),C(Y) \widehat{\otimes}\mathcal{A}(X, \mathcal{H})_{V_i}\widehat{\otimes}B)^{N_i}) \rightarrow \;\;\;\;\;\;\;\;\;\;\;\;\;\;\;\;\;\;\;\;\;\;\;\;\;\;\;\;\;\;\;\;\;\;
\end{equation*}
\begin{equation*}
 \;\;\;\;\;\;\;\;\;\;\;\;\;\;\;\;\;\;\;\;\;\;\;\;\;\;\;\;\;\;\;\;\;\;\;\;\;\;\;\;\;\;\;\;\;\;\;\;
 \displaystyle \lim_{d \to \infty}K_{\ast}(C_L^{\ast}(P_d(G), C(Y) \widehat{\otimes}\mathcal{A}(X, \mathcal{H})\widehat{\otimes}B)_{G \cdot V_i}^G),
\end{equation*}
induced by inclusion are isomorphisms.
\end{lem}

\begin{proof}
For any $T \in C^{\ast}(P_d(G), C(Y) \widehat{\otimes}\mathcal{A}(X, \mathcal{H})\widehat{\otimes}B)_{G\cdot V_i}^G$, the support of $T$ is contained in the subset
$Z_d\times Z_d \times (\bigcup_{g \in G} g V_i)$.
Since $\text{supp}(T_{y, z}) \subset  J(y) \cdot V_i$, $T$ can be expressed as a product
$$T=\prod_{g \in G/N_i} T^g,$$
where $T^g$ is a $Z_d \times Z_d$-matrix with all entries supported in $g \cdot V_i$ and the map $J: Z_d \to G$ is the map defined in Section 4.  Since every element
$$T \in C_{alg}^{\ast}(P_d(G), C(Y)\widehat{\otimes}\mathcal{A}(X, \mathcal{H})\widehat{\otimes}B)_{G \cdot V_i}^G$$
is $G$-invariant, the matrix $T$ is determined by $T^e$.

For each $g \in G$, define $A_{alg}^{g, *}$ to be the algebra of all bounded, locally compact operators $T=\big(T_{y, z}\big)_{y,z \in Z_d}$ on the Hilbert module $\ell^2(Z_d)\widehat{\otimes} C(Y) \widehat{\otimes} \mathcal{A}(X, \mathcal{H})_{g \cdot V_i} \widehat{\otimes}B \widehat{\otimes}H_G$, satisfying: (1) $\text{supp}(T_{y, z}) \subset g \cdot V_i$ for all $y, z \in P_d(G)$; (2) there exists some $C>0$, such that $T_{y, z} \neq 0$ implies that $d(y, P_d(N_i))$, $d(z, P_d(N_i)) \leq C$; (3) $T$ is $gN_ig^{-1}$-invariant.

 Taking closure under the operator norm over the Hilbert module
$$\ell^2(Z_d)\widehat{\otimes} C(Y) \widehat{\otimes} \mathcal{A}(X, \mathcal{H})_{g\cdot V_i} \widehat{\otimes}B\widehat{\otimes}H_G$$
gives rise to a $C^*$-algebra, denoted by $A^{g, *}$.

Similarly, we can define a localized version of the above algebra, denoted by $A_L^{g, *}$.
Let us define
$$A^{G,*}=\left\{(T^g) \in \prod_{g \in G/N_i} A^{g, *}: T^g=g\cdot T^e, \forall g \in G\right\},$$
and
$$A_L^{G,*}=\left\{(T^g) \in \prod_{g \in G/N_i} A_L^{g, *}: T^g=g\cdot T^e, \forall g \in G\right\}.$$

It is not difficult to check that
$$
A^{G,*}\cong C^{\ast}(P_d(G), C(Y) \widehat{\otimes}\mathcal{A}(X, \mathcal{H})\widehat{\otimes}B)_{G \cdot V_i}^G,
$$
and
$$
A_L^{G,*}\cong C_L^{\ast}(P_d(G), C(Y) \widehat{\otimes}\mathcal{A}(X, \mathcal{H})\widehat{\otimes}B)_{G \cdot V_i}^G
.$$
According to the definition of twisted Roe algebras and coarse embeddability of the group $G/N_i$, for any $T \in A^{e, *}$, there is a constant $M$ such that $T_{x,y} \neq  0$ implies that
$$x, y \in P_d(B_G(J(x)\cdot N_i, M)),$$
where $P_d(B_G(J(x)\cdot N_i, M))$ is the subcomplex of $P_d(G)$ with vertices in the $M$-neighborhood of $J(x)\cdot N_i$ in $G$.
Therefore, we have
$$A^{e, *}=\lim_{M \to \infty} A_M^{e, *}$$
and
$$A_L^{e, *}=\lim_{M \to \infty} A_{L, M}^{e, *},$$
where $A^{g,*}_M$ is the $C^*$-subalgebra of $A^{g, *}$ generated by all matrices $(T^g_{x,y})$ with $\text{supp}(T^g_{x,y}) \subset g \cdot V_i$, and $d(x, N_i), d(y,N_i)\leq M$.

For the Roe algebra case, there is a $*$-isomorphism
$$A_M^{e, *}\cong C^{\ast}(P_d(B_G(N_i, M)), C(Y) \widehat{\otimes}\mathcal{A}(X, \mathcal{H})_{V_i}\widehat{\otimes}B)^{N_i}$$
for any fixed $M>0$ by Proposition \ref{Prop_Roe_coarse_invar}, because $P_d(B_G(N_i, M))$ is $N_i$-coarsely equivalent to $P_d(N_i)$.

For the localization algebra case, it suffices to show that for $d>0$ large enough, the $C^*$-algebra $ A_{L, M}^{e, *}$ has the same $K$-theory as
$C_L^{\ast}(P_d(B_G(N_i, M)), C(Y) \widehat{\otimes}\mathcal{A}(X, \mathcal{H})\widehat{\otimes}B)^{N_i}$ for any fixed $M$. By Proposition \ref{thm_Lip_hom_inva}, it suffices to show that $P_d(B_G(N_i, M))$ is strongly $N_i$-homotopy equivalent to $P_d(N_i)$ when $d$ is large enough.

When $d$ is large enough, we can define a strong Lipschitz homotopy equivalence between the subcomplexes $P_d(B_G(N_i, M))$ and $P_d(N_i)$ as follows.
For any element $g \in G$ with $d(g, N_i)\leq M$, there exists an element $s \in G$ with $|s|\leq M$ such that $gs^{-1} \in N_i$, and the number of elements in the $M$-ball of the group $G$ is finite. Let $B_G(N_i, M)=\bigsqcup_{k=1}^{n_0}N_i s_k$ where $|s_k|\leq M$, and $\{s_1, s_2, \cdots, s_{n_0}\}$ is a subset of representatives for the right cosets $G/N_i$. We can define a map $\rho: B_G(N_i, M') \to N_i$ by $\rho(g)=g'$ where $g=g's$ is the unique product with $g' \in N_i$, and $s\in \{s_1, s_2, \cdots, s_{n_0}\}$. Uniqueness of the product is guaranteed by the fact that $\{s_1, s_2, \cdots, s_{n_0}\}$ is a subset of the representatives for the right cosets of $G/N_i$.   It is easy to check that the map $\rho$ is well-defined and $N_i$-invariant.

We can define a strong Lipschitz homotopy equivalence
$$
H(\cdot, t): P_d(B_G(N_i,  M)) \to P_d(N_i)
$$
by $$H \left( \sum_{i} c_i g_i, t\right)=\sum_{i} \big(t c_i g_i + (1-t) c_i \rho(g_i)\big),$$
where $t \in [0, 1]$,  and $\sum_i c_i g_i \in P_d(B_G(N_i, M))$. By Proposition \ref{thm_Lip_hom_inva}, the localization algebra version is done.
\end{proof}

\begin{lem}
For all $i$, and all $d>0$, we have a commutative diagram
\begin{equation*}{\small}
\xymatrix{
	K_{\ast}(C^{\ast}_L(P_d(N_i), C(Y)\widehat{\otimes}\mathcal{A}(X, \mathcal{H})_{V_i}\widehat{\otimes}B)^{N_i})\ar[d]^{(\imath_{L, i})_{\ast}}\ar[r]^{ev_{\ast}}& K_{\ast}(C^*(P_d(N_i),C(Y)\widehat{\otimes} \mathcal{A}(X, \mathcal{H})_{V_i}\widehat{\otimes}B)^{N_i})\ar[d]^{(\imath_i)_{\ast}}\\
	K_{\ast}(C^{\ast}_{L}(P_d(G),C(Y)\widehat{\otimes} \mathcal{A}(X, \mathcal{H})\widehat{\otimes}B)_{G \cdot V_i}^G)\ar[r]^{ev_{\ast}}& K_{\ast}(C^*(P_d(G),C(Y)\widehat{\otimes} \mathcal{A}(X, \mathcal{H})\widehat{\otimes}B)_{G \cdot V_i}^G).
}
\end{equation*}

\end{lem}

We need the following result due to Tu (see \cite{Tu_BC}).
 \begin{thm}\label{Thm_Tu's}
Let $Y \rtimes \Gamma$ be an a-T-menable transformation groupoid. Then for any $\Gamma$-$C^{\ast}$-algebra $A$, the Baum--Connes conjecture with coefficients in $C(Y)\widehat{\otimes} A$ holds for $\Gamma$, i.e., the map
$$ev_*: \lim _{d \to \infty} K_*(C_L^*(P_d(\Gamma), \Gamma, C(Y){\widehat\otimes} A)) \to \lim_{d\to \infty}K_*(C^*(P_d(\Gamma), \Gamma, C(Y)\widehat{\otimes} A))$$
on $K$-theory induced by the evaluation-at-zero map is isomorphic.
\end{thm}
In \cite{Tu_BC}, Tu constructed a continuous field of $C^*$-algebras which admits a proper $Y \rtimes \Gamma$-action by the a-T-menability of the groupoid $X \rtimes \Gamma$. The continuous field defined by Tu of $C^*$-algebras is essentially the same as the one that we described in Section 3. Then the Baum--Connes conjecture for the groupoid $X \rtimes \Gamma$ is reduced from the Baum--Connes conjecture for $X \rtimes \Gamma$ with coefficients in the  continuous field of $C^*$-algebras by the Dirac-dual-Dirac method. In fact, the Baum--Connes conjecture with coefficients in $C(X)$ is actually the Baum--Connes conjecture for the groupoid $X \rtimes \Gamma$.

\begin{rem}
The twisted Roe algebra $C^*(P_d(N_i), C(Y)\widehat{\otimes} \mathcal{A}(X, \mathcal{H})_{V_i})^{N_i}$ is exactly the Roe algebra for $P_d(N_i)$ with coefficients in $C(Y)\widehat{\otimes} \mathcal{A}(X, \mathcal{H})_{V_i}$ by the definition of twisted Roe algebras. Similarly, the twisted localization algebra $$C_L^*(P_d(N_i), C(Y)\widehat{\otimes} \mathcal{A}(X, \mathcal{H})_{V_i})^{N_i}$$
is the localization algebra for $P_d(N_i)$ with coefficients in $C(Y)\widehat{\otimes} \mathcal{A}(X, \mathcal{H})_{V_i}$.

\end{rem}

Combining Lemma \ref{lem_reduce_to_subgroup} with Theorem \ref{Thm_Tu's}, we have the following result.
\begin{prop}\label{thm_BC_subgroup}
The homomorphism
\begin{equation*}
ev_*:\displaystyle \lim_{d \to \infty}K_*(C^*_L(P_d(G), C(Y) \widehat{\otimes} \mathcal{A}(X, \mathcal{H})\widehat{\otimes}B)_{G \cdot V_i}^G) \to  \;\;\;\;\;\;\;\;\;\;\;\;\;\;\;\;\;\;\;\;\;\;\;\;\;\;\;\;\;\;\;\;\;\;\;\;\;\;\;\;\;\;\;\;\;\;
\end{equation*}
\begin{equation*}
\;\;\;\;\;\;\;\;\;\;\;\;\;\;\;\;\;\;\;\;\;\;\;\;\;\;\;\;\;\;\;\;\;\;\;\;\;\;\;\;\;\;\;\;\;
\displaystyle \lim_{d \to\infty}K_*(C^*(P_d(G), C(Y) \widehat{\otimes} \mathcal{A}(X, \mathcal{H})\widehat{\otimes}B)_{G \cdot V_i}^G)
\end{equation*}
 induced by evaluation-at-zero is an isomorphism.
\end{prop}

\begin{lem}
Let $O$ and $O'$ be open subsets of $\mathbb{R}_+ \times \mathcal{H}$. If $O \subset O'$,
then the $C^*$-subalgebras $C^*(P_d(G), C(Y)\widehat{\otimes}\mathcal{A}(X, \mathcal{H})\widehat{\otimes}B)_O^G$ and $C_L^*(P_d(G),C(Y)\widehat{\otimes}\mathcal{A}(X, \mathcal{H})\widehat{\otimes}B)_O^G$ are respectively closed, two-sided ideals of $C^*$-algebras
$C^*(P_d(G), C(Y)\widehat{\otimes}\mathcal{A}(X, \mathcal{H})\widehat{\otimes}B)_{O'}^G$ and
$C_L^*(P_d(G), C(Y)\widehat{\otimes}\mathcal{A}(X, \mathcal{H})\widehat{\otimes}B)_{O'}^G$.
\end{lem}

In the definition of $V_i$, the choice of $V_i$ depends on the constant $R$, so we denote it as $V_i(R)$ in the rest of this section. Fixing any $R_0$, for every $R<R_0$, let $V_{R, i}=V_i(R_0) \cap O_R(e)$ where $\left\{V_i(R_0)\right\}_{i=1}^N$ is a finite open cover of $O_R(e)$ obtained from the compactness of $\widebar{O_{R_0+1}}(e)$. We obtain the following decomposition for the twisted localization algebras and the twisted Roe algebras.

\begin{lem}\label{decomposeRoe}
For every fixed $R_0>0$, let $O_{R,i_0}=\bigcup_{i=1}^{i_0} G \cdot V_i(R)$.
Then we have the following,

(1) $C^*(P_d(G), C(Y) \widehat{\otimes}\mathcal{A}(X, \mathcal{H})\widehat{\otimes}B)_{O_{R, i_0+1}}^G=
    C^*(P_d(G), C(Y) \widehat{\otimes}\mathcal{A}(X, \mathcal{H})\widehat{\otimes}B)_{O_{R, i_0}}^G\\
+C^*(P_d(G), C(Y) \widehat{\otimes}\mathcal{A}(X, \mathcal{H})\widehat{\otimes}B)_{ G \cdot V_{i_0+1}(R)}^G$,\\

(2) $C_L^*(P_d(G), C(Y) \widehat{\otimes}\mathcal{A}(X, \mathcal{H})\widehat{\otimes}B)_{O_{R, i_0+1}}^G=
    C_L^*(P_d(G), C(Y) \widehat{\otimes}\mathcal{A}(X, \mathcal{H})\widehat{\otimes}B)_{O_{R, i_0}}^G\\
+C_L^*(P_d(G), C(Y) \widehat{\otimes}\mathcal{A}(X, \mathcal{H})\widehat{\otimes}B)_{G \cdot V_{i_0+1}(R)}^G$,\\

(3) $C^*(P_d(G), C(Y) \widehat{\otimes}\mathcal{A}(X, \mathcal{H})\widehat{\otimes}B)_{O_{R, i_0}\cap G \cdot V_{i_0+1}}^G=C^*(P_d(G), C(Y) \widehat{\otimes}\mathcal{A}(X, \mathcal{H})\widehat{\otimes}B)_{O_{R, i_0} }^G\\
 \cap C^*(P_d(G), C(Y) \widehat{\otimes}\mathcal{A}(X, \mathcal{H})\widehat{\otimes}B)_{ G \cdot V_{i_0+1}}^G$,\\

(4) $C_L^*(P_d(G), C(Y) \widehat{\otimes}\mathcal{A}(X, \mathcal{H})\widehat{\otimes}B)_{O_{R, i_0}\cap G \cdot V_{i_0+1}(R)}^G=C_L^*(P_d(G), C(Y) \widehat{\otimes}\mathcal{A}(X, \mathcal{H})\widehat{\otimes}B)_{O_{R, i_0} }^G\\
 \cap C_L^*(P_d(G), C(Y) \widehat{\otimes}\mathcal{A}(X, \mathcal{H})\widehat{\otimes}B)_{G \cdot V_{i_0+1}(R)}^G$,\\
 for all $R \leq R_0$.
\end{lem}

The proof of Lemma \ref{decomposeRoe} is similar to that of Lemma 6.3 in \cite{Yu_coa_emb}, and is therefore omitted.

Note that, for any $1 \leq i < i' \leq N$, the intersection $O_R(i)\cap O_R(i')$ is of the same form as $V'\times_{N'} G$ for some open subset $V' \subset W$ and some subgroup $N'\leq G$ containing $N$ with finite index. So by the Mayer--Vietoris sequence and the Five Lemma, we have the following result.

\begin{prop}
Let $1 \to N \to G \to G/N \to 1$ be an extension of countable discrete groups. Assume $N$ and $G/N$ are coarsely embeddable into Hilbert spaces. The evaluation-at-zero map
\begin{equation*}
ev_{\ast}:\displaystyle \lim_{d \to \infty}K_{\ast}(C_L^{\ast}(P_d(G),C(Y)\widehat{\otimes} \mathcal{A}(X, \mathcal{H})\widehat{\otimes}B)_{O_R}^G) \rightarrow \;\;\;\;\;\;\;\;\;\;\;\;\;\;\;\;\;\;\;\;\;\;\;\;\;\;\;\;\;\;\;
\end{equation*}
\begin{equation*}
\;\;\;\;\;\;\;\;\;\;\; \displaystyle \lim_{d \to \infty}K_{\ast}(C^{\ast}(P_d(G),C(Y)\widehat{\otimes} \mathcal{A}(X, \mathcal{H})\widehat{\otimes}B)_{O_R}^G)
 \end{equation*}
is isomorphic for each $R$.
\end{prop}

By taking the direct limit over $R$, we have the twisted Baum--Connes conjecture with coefficients in
$C(Y) \widehat{\otimes}\mathcal{A}(X, \mathcal{H})\widehat{\otimes} B$.

\begin{proof}[The proof of Theorem \ref{thm_twist_BC}]
Since we have
$$
C_L^*(P_d(G), C(Y)\widehat{\otimes}\mathcal{A}(X, \mathcal{H})\widehat{\otimes}B)^G=\lim_{R\to \infty}C_L^*(P_d(G), C(Y)\widehat{\otimes}\mathcal{A}(X, \mathcal{H})\widehat{\otimes}B)_{O_R}^G,
$$
and
$$
C_L^*(P_d(G), C(Y)\widehat{\otimes}\mathcal{A}(X, \mathcal{H})\widehat{\otimes}B)^G=\lim_{R\to \infty}C_L^*(P_d(G), C(Y)\widehat{\otimes}\mathcal{A}(X, \mathcal{H})\widehat{\otimes}B)_{O_R}^G
,$$
it suffices to show that the map
$$ ev_*: \displaystyle{\lim_{d \to \infty}K_*(D_{d, L,R_0}) \to \lim_{d \to \infty}K_*(D_{d, R_0})},$$
is an isomorphism,
where we set $$D_{d, L, R_0}=\lim_{R<R_0, R \to R_0}C_L^*(P_d(G), C(Y)\widehat{\otimes}\mathcal{A}(X, \mathcal{H})\widehat{\otimes}B)_{O_R}^G$$ and
$$D_{d, R_0}=\lim_{R<R_0, R \to R_0}C^*(P_d(G), C(Y)\widehat{\otimes}\mathcal{A}(X, \mathcal{H})\widehat{\otimes}B)_{O_R}^G$$
for brevity.
By Lemma \ref{decomposeRoe} and $O_R=\bigcup_i G\cdot V_i(R)$, the proof is completed using the Mayer--Vietoris sequence and the Five Lemma.
\end{proof}

 \section{Proof of the main theorem}

 In this section, we will prove the main result of this paper. We are going to define a geometric analogue of the infinite-dimensional Bott map introduced by Higson, Kasparov and Trout in \cite{KaHiTr}, and then prove the Novikov conjecture from the twisted Baum--Connes conjecture.

 Let $X$ be the compact space and $\big(\mathcal{H}_x\big)_{x \in X}$ the continuous field of Hilbert spaces defined in Section 4. Every element in $C(X) \widehat{\otimes} \mathcal{S}$ can be viewed as a continuous $\mathcal{S}$-valued map. The action of $G$ on $C(X)\widehat{\otimes} \mathcal{S}$ is given by $g\cdot \big(f_x\big)_{x \in X}=\big(f'_x\big)_{x \in X}$, where $f'_x=f_{xg^{-1}}$.

In Section 3, we obtain the fiber-wise defined Bott map $\beta_t: C(X)\widehat{\otimes} \mathcal{S} \to \mathcal{A}(X, \mathcal{H})$, by $$\beta_t((f_x)_{x \in X})=(\beta_t(f_x))_{x \in X}$$
where $f_x \in \mathcal{S}$, for any $x \in X$. It induces an isomorphism on the $K$-theory level.

Each element $T \in C^{\ast}(P_d(G), G, C(Y) \widehat{\otimes}C(X) \widehat{\otimes}B) \widehat{\otimes} \mathcal{S}$ can be expressed as
$$T=\big(T_{y, z}\big)_{y, z \in Z_d} \widehat{\otimes}f,$$
where $f \in \mathcal{S}$, $T_{y, z} \in C(Y) \widehat{\otimes}C(X)\widehat{\otimes} B \widehat{\otimes} K_G$, for all $y, z \in Z_d$.

We define the Bott map between the algebraic Roe algebras and algebraic twisted localization algebras,
$$\beta: C_{alg}^{\ast}(P_d(G), G, C(Y)\widehat{\otimes}C(X)\widehat{\otimes}B) \widehat{\otimes} \mathcal{S} \rightarrow C_{alg}^{\ast}(P_d(G),C(Y)\widehat{\otimes} \mathcal{A}(X, \mathcal{H})\widehat{\otimes}B)^G,$$
by $$\beta_t (T)(y, z)=T_{y, z}\widehat{\otimes} s_{y, z},$$
where $s_{y, z} \in \mathcal{A}(X, \mathcal{H})$ is a section with $s_{y, z}(x)=\beta_t( b(x, J(y))(f)$, and $$\beta_t(b(x, J(y))): \mathcal{A}(\{b(x, J(y))\})\to \mathcal{A}(\mathcal{H}_x)$$
is the Bott map induced by the inclusion $\{b(x, J(y))\} \to \mathcal{H}_x$, for $t \in [0, \infty)$.

We are going to show that the Bott map $\beta_t$ is well-defined. It suffices to show that, for each
$T\widehat{\otimes} f=\big(T_{y, z} \big)_{y, z \in Z_r} \widehat{\otimes} f \in C_{alg}^{\ast}(P_d(G), G, C(Y)\widehat{\otimes}C(X)\widehat{\otimes}B) \otimes \mathcal{S}$ ,
$\beta_t(T \widehat{\otimes}f)$ is $G$-invariant in the sense that $g(\beta_t(T\widehat{\otimes}f)_{g^{-1}y, g^{-1}z})=\beta_t(T\widehat{\otimes}f)_{y, z}$ for all $g \in G$, $y, z \in Z_d$.

Let $s_{g}: X \to \mathcal{H}=\bigsqcup_{x \in X} \mathcal{H}_x$ be the continuous section defined by $s_g(x)=b(x, g)$, for all $g \in G$. By Lemma \ref{lem,diagr_for_action}, we have that
 $\beta(g \cdot s)=g \cdot\beta(s)$ for all $s\in \mathcal{A}(X, \mathcal{ H})$. For each $g \in G$, $y, z \in Z_d$,
 it follows from the definition of $T \widehat{\otimes}f$ that
 $$g(T_{g^{-1}y, g^{-1}z})=T_{y, z}.$$
So we need to show:
$$g \cdot( \beta_t(T\widehat{\otimes}f)_{g^{-1}y, g^{-1}z})=\beta_t(T \widehat{\otimes}f)_{y, z}.$$
On the one hand,
$$
\beta_t(T \widehat{\otimes}f)_{y, z} = T_{y, z} \widehat{\otimes} \beta_t(s_{J(z)})(f).
$$
On the other hand, we have
$$
g \cdot T_{g^{-1}y, g^{-1}z} \widehat{\otimes} g \cdot(\beta_t(s_{J(g^{-1}y)})(f))=T_{y, z} \widehat{\otimes} g \cdot(\beta_t(s_{J(g^{-1}y)})(f)).
$$

It suffices to show that $g \cdot(\beta_t(s_{J(g^{-1}y)})(f))=\beta(s_{J(y)})(f)$.
For each $x \in X$, we have $\beta(s_{J(y)})(f)(x)=\beta (b(x, J(y)))(f) \in \mathcal{A}(\mathcal{H}_x)$.

 For the section $s_{J(y)}$, we have
\begin{align*}
 (g \cdot s_{J(g^{-1}y)})(x) &=g(s_{J(g^{-1}y)}(xg))\\
 &=g(b(xg, g^{-1}J(y)))\\
 &=b(x, J(y))\\
 &=s_{J(y)}(x),
 \end{align*}
 for all $x \in X$. Therefore $g \cdot s_{J(g^{-1}y)}=s_{J(y)}$, for all $y \in Z_d$, and hence it follows that $$g(\beta(s_{J(g^{-1}y)})(f))=\beta_t{s_{J(y)}}(f).$$
 Therefore,
 $$g(\beta_t (T\widehat{\otimes}f)_{g^{-1}y, g^{-1}z})=  \beta_t(T\widehat{\otimes}f)_{y, z}.$$
 As a consequence, the map
$$\beta_t: C_{alg}^*(P_d(G), G, C(Y)\widehat{\otimes}C(X)) \widehat{\otimes} \mathcal{S}\to C_{alg}^*(P_d(G), C(Y)\widehat{\otimes}\mathcal{A}(X, \mathcal{H}))^G$$
for all $t \in [1, \infty)$ is well-defined.

\begin{prop}
  The family of maps $$\beta_t: C_{alg}^*(P_d(G), G, C(Y)\widehat{\otimes}C(X)\widehat{\otimes}B) \widehat{\otimes} \mathcal{S}\to C_{alg}^*(P_d(G), C(Y)\widehat{\otimes}\mathcal{A}(X, \mathcal{H})\widehat{\otimes}B)^G$$
 for $t\in [1, \infty)$ extends to an asymptotic morphism
  $$\beta: C^*(P_d(G), G, C(Y)\widehat{\otimes}C(X)\widehat{\otimes}B) \widehat{\otimes} \mathcal{S}\to C^*(P_d(G), C(Y)\widehat{\otimes}\mathcal{A}(X, \mathcal{H})\widehat{\otimes}B)^G.$$
\end{prop}
\begin{proof}
  Let
  $$E=\left\{\sum_{z \in Z_d}a_z [z]: a_z \in C(Y)\widehat{\otimes}\mathcal{A}(X, \mathcal{H})\widehat{\otimes}B\widehat{\otimes}K_G, \sum_{z \in Z_d}a_z^*a_z \text{~converges~in~norm}\right\}.$$
For every $g \in \mathcal{S}$, define a bounded module homomorphism $N_g: E \to E$ given by
$$
N_g\left(\sum_{z \in Z_d}a_z [z]\right)=\sum_{z \in Z_d}(\beta(s(J(z)))(g)\widehat{\otimes}1) a_z [z]
$$
for all $\sum_{z \in Z_d}a_z [z]$.
It is easy to check that
$$\beta_t(T \widehat{\otimes}g)=N_{g_t}(T \widehat{\otimes}1)$$
for all $g \in \mathcal{S}$ and $T \in C^*_{alg}(P_d(G), G, C(Y)\widehat{\otimes} C(X)\widehat{\otimes}B)$, where $T \widehat{\otimes}1$ is a bounded module homomorphism from $E$ to $E$ by
$$
(T \widehat{\otimes}1)\left(\sum_{z \in Z_d}a_z [z]\right)=\sum_{y \in Z_d} \left(\sum_{z}(T_{y, z}\widehat{\otimes}1)a_z\right)[y].
$$
By the definition of Roe algebras, the map $\beta_t$ extends to a linear map
$$C^*(P_d(G), G, C(Y) \widehat{\otimes}C(X)\widehat{\otimes}B)\widehat{\otimes}_{alg}\mathcal{S} \to C^*(P_d(G), C(Y)\widehat{\otimes}\mathcal{A}(X, \mathcal{H})\widehat{\otimes}B)^G$$
satisfying
$$\|\beta_t(T \widehat{\otimes}g)\|\leq \|g\|\|T\|$$
for all $g \in \mathcal{S}$, and $T \in C_{alg}^*(P_d(G), G, C(Y)\widehat{\otimes}C(X)\widehat{\otimes}B)$.
By the definition of $$C^*(P_d(G), G, C(Y)\widehat{\otimes} C(X) \widehat{\otimes}B)$$ and the properness of the $G/N$-action on the $C^*$-algebra $C(Y)\widehat{\otimes}\mathcal{A}(X, \mathcal{H})\widehat{\otimes}B$, one can verify that $\beta_t$ is an asymptotic morphism from $C^*(P_d(G), G, C(Y) \widehat{\otimes}C(X)\widehat{\otimes} B)\widehat{\otimes}_{alg}\mathcal{S}$ to $C^*(P_d(G), C(Y)\widehat{\otimes}\mathcal{A}(X, \mathcal{H})\widehat{\otimes}B)^G$. Hence $\beta_t$ extends to a homomorphism
$$C^*(P_d(G), G, C(Y) \widehat{\otimes}C(X)\widehat{\otimes}B)\widehat{\otimes}_{\operatorname{max}}\mathcal{S} \to C^*(P_d(G), C(Y)\widehat{\otimes}\mathcal{A}(X, \mathcal{H})\widehat{\otimes}B)^G.$$
As a consequence of the nuclearity of $\mathcal{S}$, $\beta_t$ extends to a homomorphism
\begin{align*}
C^*(P_d(G), G, C(Y) \widehat{\otimes}C(X)\widehat{\otimes}B)\widehat{\otimes}\mathcal{S}&\to
C^*(P_d(G), C(Y)\widehat{\otimes}\mathcal{A}(X, \mathcal{H})\widehat{\otimes}B)^G.\qedhere
\end{align*}
\end{proof}
We define the Bott map on $K$-theory,
$$\beta_*: K_*(C^*(P_d(G), G, C(Y)\widehat{\otimes}C(X)\widehat{\otimes}B) \widehat{\otimes} \mathcal{S}))\to K_*(C^*(P_d(G), C(Y)\widehat{\otimes}\mathcal{A}(X, \mathcal{H})\widehat{\otimes}B)^G),$$
as that induced by the asymptotic morphism
$$\beta_t: C^*(P_d(G), G, C(Y)\widehat{\otimes}C(X)\widehat{\otimes}B) \widehat{\otimes} \mathcal{S})\to C^*(P_d(G), C(Y)\widehat{\otimes}\mathcal{A}(X, \mathcal{H})\widehat{\otimes}B)^G$$
for $t \in [1, \infty)$.

Similarly,  we can define the localized version of the asymptotic morphism
$$\beta_{L,t}: C_L^*(P_d(G), G, C(Y)\widehat{\otimes}C(X)\widehat{\otimes}B) \widehat{\otimes} \mathcal{S}\to C_L^*(P_d(G), C(Y)\widehat{\otimes}\mathcal{A}(X, \mathcal{H})\widehat{\otimes}B)^G.$$
for $t \in [1, \infty)$.
Since this is an asymptotic morphism, it induces the Bott map on $K$-theory
$$\beta_{L,*}: K_*(C_L^*(P_d(G), G, C(Y)\widehat{\otimes}C(X)\widehat{\otimes}B) \widehat{\otimes} \mathcal{S})\to K_*(C_L^*(P_d(G), C(Y)\widehat{\otimes}\mathcal{A}(X, \mathcal{H})\widehat{\otimes}B)^G).$$
Note that, for every $d>0$, the following diagram
\begin{equation*}
\xymatrix{
	K_{\ast}(C^{\ast}_L(P_d(G), G, C(Y)\widehat{\otimes}C(X)\widehat{\otimes}B)\widehat{\otimes} \mathcal{S})\ar[d]^{\beta_{L,*}}\ar[r]^{ev_*}& K_{\ast}(C^*(P_d(G), G, C(Y)\widehat{\otimes}C(X)\widehat{\otimes}B)
\widehat{\otimes}\mathcal{S})\ar[d]^{\beta_*}\\
	K_{\ast}(C^{\ast}_L(P_d(G), C(Y)\widehat{\otimes}\mathcal{A}(X, \mathcal{H})\widehat{\otimes}B)^G)\ar[r]^{ev_*}& K_{\ast}(C^*(P_d(G), C(Y)\widehat{\otimes} \mathcal{A}(X, \mathcal{H})\widehat{\otimes}B)^G).
}
\end{equation*}
is commutative.

Let $C^*_L(P_d(G), C(Y) \widehat{\otimes} \mathcal{A}(X, \mathcal{H})\widehat{\otimes}B)^G$ be the twisted localization algebra. Let $C_1\subset C_2 \subset P_d(G)$ be $G$-invariant closed subsets such that $C_i=\widebar{C_i\cap Z_d}$ for $i=1, 2$. Assume that the inclusion map $i: C_1 \to C_2$ is a strong Lipschitz homotopy equivalence. Let $C^*_L(C_i, C(Y) \widehat{\otimes} \mathcal{A}(X, \mathcal{H})\widehat{\otimes}B)^G$ be the $C^*$-subalgebra consisting of all the operators $T \in C^*_L(P_d(G), C(Y) \widehat{\otimes} \mathcal{A}(X, \mathcal{H})\widehat{\otimes}B)^G$ with support in $C_i \times C_i \times W \subset Z_d \times Z_d \times W $, for $i=1,2$.

Then the map $i: C_1 \hookrightarrow C_2$ induces a map on the twisted Roe algebras
$$i_*: C^*(C_1, C(Y) \widehat{\otimes} \mathcal{A}(X, \mathcal{H})\widehat{\otimes}B)^G \to C^*(C_2, C(Y) \widehat{\otimes} \mathcal{A}(X, \mathcal{H})\widehat{\otimes}B)^G$$
by

\[ i(T)_{y, z} =
 \begin{cases*}
    T_{y, z} \quad& if $y, z \in C_1$,\\
    0 & otherwise,  \\
\end{cases*}\]
for all $T=(T_{y, z})_{y, z \in C_1} \in C^*(C_1, C(Y) \widehat{\otimes} \mathcal{A}(X, \mathcal{H})\widehat{\otimes}B)^G$.
This map is well-defined, because the inclusion map $i: C_1 \hookrightarrow C_2$ is isometric. Similarly, one can define a homomorphism between the localization algebras
$$i_L: C^*_L(C_1, C(Y) \widehat{\otimes} \mathcal{A}(X, \mathcal{H})\widehat{\otimes}B)^G \to C^*_L(C_2, C(Y) \widehat{\otimes} \mathcal{A}(X, \mathcal{H})\widehat{\otimes}B)^G.$$
The following result is a twisted analogue of the result that the $K$-theory of twisted localization algebras is invariant under strong Lipschitz homotopy equivalence. The proof is similar to that of Proposition \ref{thm_Lip_hom_inva}.
\begin{lem}\label{thm_twisted_strongly_homo_invar}
Assume the inclusion map $i: C_1 \hookrightarrow C_2$ is a strong Lipschitz homotopy equivalence. The map
$$(i_L)_*: K_*(C^*_L(C_1, C(Y) \widehat{\otimes} \mathcal{A}(X, \mathcal{H})\widehat{\otimes}B)^G) \to K_*(C^*_L(C_2, C(Y) \widehat{\otimes} \mathcal{A}(X, \mathcal{H})\widehat{\otimes}B)^G)$$
 induced by $i_L$ on $K$-theory is an isomorphism.
\end{lem}

Similarly, we have a twisted version of the Mayer--Vietoris sequence for twisted localization algebras, as in Proposition \ref{prop_MV squence_Local}. The proof will be similar, so it is omitted.
\begin{lem}\label{lem_twisted_MV squence_Local}
Let $\Delta$ be a simplicial complex endowed with the spherical metric, and let $G$ be a countable discrete group. Assume $G$ acts on $\Delta$ properly by isometries. Let $C_1, C_2\subset \Delta$ be $G$-invariant simplicial subcomplexes endowed with the subspace metric. Then we have the following six-term exact sequence

\begin{equation*}
\begin{tikzcd}
	K_0(L_{C_1\cap C_2})\arrow[r] & K_0(L_{C_1}) \oplus K_0(L_{C_2})  \arrow[r] & K_0(L_{C_1\cup C_2})\arrow[d]\\
	K_1(L_{C_1\cap C_2})\arrow[u]\arrow[r] & K_1(L_{C_1}) \oplus K_1(L_{C_2})\arrow[r]& K_1(L_{C_1\cup C_2}),
\end{tikzcd}
\end{equation*}
where we set
 $L_{C_1}=C_L^*(X_1, C(Y) \widehat{\otimes} \mathcal{A}(X, \mathcal{H})\widehat{\otimes}B)^G$, $L_{C_2}=C_L^*(X_1, C(Y) \widehat{\otimes} \mathcal{A}(X, \mathcal{H})\widehat{\otimes}B)^G$, $L_{C_1\cap C_2}=C_L^*(X_1\cap X_1, C(Y) \widehat{\otimes} \mathcal{A}(X, \mathcal{H})\widehat{\otimes}B)^G$ and $L_{C_1 \cup C_2}=C_L^*(X_1 \cup C_2, C(Y) \widehat{\otimes} \mathcal{A}(X, \mathcal{H})\widehat{\otimes}B)^G$
 for short.
\end{lem}

\begin{prop}
For each $d>0$, the map
$$\beta_{L,*}: K_*(C_L^*(P_d(G), G, C(Y)\widehat{\otimes}C(X)\widehat{\otimes}B) \widehat{\otimes} \mathcal{S})\to K_*(C_L^*(P_d(G), C(Y)\widehat{\otimes}\mathcal{A}(X, \mathcal{H})\widehat{\otimes}B)^G)$$
is an isomorphism.
\end{prop}
\begin{proof}
 By induction on the dimension of the space $P_d(G)$, the theorem is a consequence of the fact that the $K$-theory of localization algebras is invariant under the strong Lipschitz homotopy equivalence (see Lemma \ref{thm_twisted_strongly_homo_invar}), Theorem \ref{thm,eq_Bott_isom_fini_gro}, the twisted Mayer--Vietoris sequence (see Lemma \ref{lem_twisted_MV squence_Local}) and the Five Lemma.
\end{proof}

Let us recall the commutative diagram which is obvious from the definition of twisted localization algebras and twisted Roe algebras.

\begin{equation*}{\small}
\xymatrix{
K_{\ast}(C^{\ast}_L(P_d(G), G, C(Y)\widehat{\otimes}C(X)\widehat{\otimes}B)\widehat{\otimes}\mathcal{S}))\ar[d]^{\beta_{L,*}}\ar[r]^{ev_{\ast}}&
K_{\ast}(C^*(P_d(G), G, C(Y)\widehat{\otimes} C(X)\widehat{\otimes}B) \widehat{\otimes}\mathcal{S})\ar[d]^{\beta_*}\\
K_{\ast}(C^{\ast}_{L}(P_d(G),C(Y)\widehat{\otimes} \mathcal{A}(X, \mathcal{H})\widehat{\otimes}B)^G)\ar[r]^{ev_{\ast}}& K_{\ast}(C^*(P_d(G),C(Y)\widehat{\otimes} \mathcal{A}(X, \mathcal{H})\widehat{\otimes}B)^G).
}
\end{equation*}

In the above diagram, the vertical map $\beta_{L,*}$ and the bottom horizontal map are isomorphisms. As a result, we have that the Novikov conjecture holds  for $G$ with coefficients in $C(Y) \widehat{\otimes} C(X)\widehat{\otimes}B$.

\begin{lem}\label{lem_BC_for_C(Y)timesC(X)}
  The map
  $$
  ev_*: \displaystyle\lim_{d \to \infty}K_{\ast}(C^{\ast}_L(P_d(G), G, C(Y)\widehat{\otimes}C(X)\widehat{\otimes}B)) \to \displaystyle\lim_{d \to \infty}K_{\ast}(C^{\ast}(P_d(G), G, C(Y)\widehat{\otimes}C(X)\widehat{\otimes}B)),
  $$
 induced by the evaluation-at-zero map on $K$-theory is injective.
\end{lem}

In the rest of this section, we will reduce the Novikov conjecture to Lemma \ref{lem_BC_for_C(Y)timesC(X)}.
By identifying the $C^*$-algebra $C(Y) \widehat{\otimes}C(X)$ with $C(Y \times X)$, one can define a map $c: \mathbb{C} \to C(Y) \widehat{\otimes} C(X)$ mapping each constant $s \in \mathbb{C}$ to the constant function with value $s$ on $Y \times X$. By tensoring with the identity map on compact operators,  we can define a map $c\widehat{\otimes} 1: K_G \to C(Y) \widehat{\otimes} C(X) \widehat{\otimes}K_G$.

Note that $Y \times X$ admits a $G$-action and it is $G_0$-contractible for any finite subgroup $G_0 \leq G$.
The map $c\otimes 1$ induces a homomorphism between the Roe algebras $$\widetilde{c}: C^*(P_d(G), G, B) \to C^*(P_d(G), G, C(Y) \widehat{\otimes} C(X)\widehat{\otimes}B)$$ given by
$$\widetilde{c}(T)(x,y)=c\widehat\otimes1 (T_{x,y}),$$
for $x,y \in Z_d$, $T=\big(T_{x,y}\big)_{x,y \in Z_d}$.

Similarly, one can define a localized version of the homomorphism
$$\widetilde{c}_L: C_L^*(P_d(G), G, B) \to C_L^*(P_d(G), G, C(Y)\widehat{\otimes}C(X)\widehat{\otimes}B),$$
by $$\widetilde{c}_L(g)(t)=\widetilde{c}(g(t)),$$
where $g \in C_L^*(P_d(G), G, B)$.

\begin{lem}\label{lem_localization_infinite_to_finite}
  Let $G_0\leq G$ be a finite subgroup, and $B$ any $G$-$C^*$-algebra. If $V\subset P_d(G)$ is a $G_0$invariant and  $G_0$-contractible subcomplex such that $G\cdot V$ is homeomorphic to the space $V \times_{G_0} G$, then we have $$K_*(C_L^*(V \times_{G_0} G, G, B)) \cong K_*(C_L^*(V, G_0, B)).$$
\end{lem}

\begin{proof}
Define a homomorphism
$$\imath: C_L^*(V, G_0, B) \to C_L^*(V_1 \times_{ G_0} G, G, B)$$
by
\[ \imath(b(t))_{x,y} =
 \begin{cases*}
    g^{-1}(b(t))_{gx,gy} \quad& if $\exists $ $g \in G  $  such that $gx, gy \in V$,\\
    0 & otherwise,
\end{cases*}\]
for all $b(t) \in C_L^*(V, G_0, B)$. Any element in $K_1(C_L^*(V \times_{G_0}G, G, B)) $ can be represented by an invertible element $a+I$, for some $a \in C_L^*(V \times_{G_0} G, G, B)$. Since the propagation of $a(t)$ approaches $0$ as $t \to \infty$, we can find a large constant $T_0$ such that $\text{supp}(a(t)) \subset \bigsqcup_{g \in G/G_0} gV \times gV$ for all $t \geq T_0$. By uniform continuity of the path $a(t)$, $a_s(t) = a(t+ s T_0)$ ($s \in [0, 1]$) is a homotopy between $a(t)$ and $a(t +T_0)$. Thus, any element in $K_1(C_L^*(V \times_{G_0} G, G, B))$ can be represented by an invertible element $b+I \in \big(C_L^*(V \times_{G_0} G, G, B)\big)^+$ with $\text{supp}(b(t)) \subset \bigsqcup_{g \in G/G_0} gV_1 \times gV_1$. Since $b(t)$ is $G$-invariant for all $t \in [0, \infty)$, we can find an element $b' \in C_L^*(V, G_0, B)$ such that $\imath_*([b' + I])=[b +I]$. Consequently, the map $\imath$ is onto, and we can similarly prove that it is injective. Therefore we have $$K_1(C_L^*(V \times_{G_0} G, G, B)) \cong K_1(C_L^*(V, G_0, B)).$$ The $K_0$ case can be dealt with by a suspension argument.
\end{proof}

The following result was proved originally using $E$-theory (\cite{GHT}, Lemma 12.11). We now give an alternative proof using localization algebras.
\begin{prop}
The map
$$(\widetilde{c}_L)_*:\lim_{d \to \infty} K_*(C_L^*(P_d(G), G, B)) \to \lim_{d \to \infty}
K_*(C_L^*(P_d(G), G, C(Y)\widehat{\otimes}C(X)\widehat{\otimes}B))$$
on $K$-theory induced by $\widetilde{c}_L$ is an isomorphism.
\end{prop}
 \begin{proof}
 When $d$ is large enough, there exist finitely many precompact open subsets $V_i$, and finite subgroups $G_i \leq G$, $i=1, \ldots, k$, such that $P_d(G)=\bigcup_{i=1}^{k} V_i \times_{G_i} G$ and each $V_i$ is $G_i$-contractible by a strong Lipschitz $G_i$-homotopy equivalence. The existence of subsets $V_i$ and finite subgroups $G_i$ is guaranteed by properness of the $G$-action on $P_d(G)$.
By the Mayer--Vietoris sequence, it suffices to show that the map $(\widetilde{c}_L)_*: K_*(C_L^*(V_i\times_{G_i}G, G, B))\to K_*(C_L^*(V_i\times_{G_i}G, G, C(Y \times X)\widehat{\otimes}B))$ is an isomorphism for each $i$. Without loss of generality, it suffices to show this for $i=1$.

By Lemma \ref{lem_localization_infinite_to_finite}, we have
$$K_*(C_L^*(V_1 \times_{G_1} G, G, B)) \cong K_1*C_L^*(V_1, G_1, B)),$$
 and
$$K_*(C_L^*(V_1 \times_{G_1} G, G, C(Y \times X)\widehat{\otimes}B)) \cong K_*(C_L^*(V_1, G_1, C(Y \times X)\widehat{\otimes}B)).$$
It suffices to show that
$$K_*(C_L^*(V_1, G_1, B))\cong K_*(C_L^*(V_1, G_1, C(Y \times X)\widehat{\otimes}B)).$$
By the $G_1$-contractibility of $V_1$, we have
$$K_*(C_L^*(V_1, G_1, B))\cong K_*(C_L^*(pt, G_1, B)),$$
and
$$K_*(C_L^*(V_1, G_1, C(Y \times X)\widehat{\otimes}B))\cong K_*(C_L^*(pt, G_1, C(Y \times X)\widehat{\otimes}B)).$$
It suffices to show that $K_*(C_L^*(pt, G_1, B)\cong K_*(C_L^*(pt, G_1, C(Y \times X)\widehat{\otimes}B)$.
Let $c: \mathbb{C} \to C(Y \times X)$ be the $*$-homomorphism given by mapping each constant $s \in \mathbb{C}$ to the constant function with value $s$ on $Y \times X$. Let $c': C(Y \times X) \to \mathbb{C}$ be the $*$-homomorphism obtained by evaluation at the point $(y_0, x_0)$, where the space $Y \times X$ contracts to a point $(y_0, x_0) \in Y \times X$ via some $G_1$-equivariant homotopy equivalence.  So we have the $*$-homomorphisms on the reduced crossed product
$$c: B\rtimes_{\operatorname{r}} G_1 \to (C(Y \times X)\widehat{\otimes}B) \rtimes_{\operatorname{r}} G_1$$
and
$$c': (C(Y \times X)\widehat{\otimes}B)\rtimes_{\operatorname{r}} G_1 \to B \rtimes_{\operatorname{r}} G_1.$$
Note that, on the $K$-theory level, the compositions $c_* \circ c'_*$ and $c'_* \circ c_*$ are identity maps.
By the Green-Julg Theorem (see \cite{GHT}),
$$ev_*: K_*(C_L^*(pt, G_1, B)) \to K_*(C^*(pt, G_1, B))$$
and
$$
ev_*:K_*(C_L^*(pt, G_1, C(Y \times X)\widehat{\otimes}B))\to K_*(C^*(pt, G_1, C(Y \times X)\widehat{\otimes}B))
$$
are isomorphisms on $K$-theory, where $ev_*$ is the homomorphism induced by the evaluation-at-zero map.
Since we have
$$C^*(pt, G_1, B)\cong (B \rtimes_{\operatorname{r}}G_1) \otimes \mathcal{K},$$
and
$$C^*(pt, G_1, C(Y \times X)\widehat{\otimes}B)\cong ((C(Y \times X)\widehat{\otimes}B) \rtimes_{\operatorname{r}}G_1) \otimes \mathcal{K}.$$
In addition, the map
$$(\widetilde{c}_L)_*:K_*(C_L^*(pt, G_1, B))\to K_*(C_L^*(pt, G_1, C(Y \times X)\widehat{\otimes}B))
$$
is an isomorphism.
As a result,
$$(\widetilde{c}_L)_*:K_*(C_L^*(V_i\times_{G_i}G, G, B))\to K_*(C_L^*(V_i, G, C(Y \times X)\widehat{\otimes}B))
$$
is an isomorphism for each $i$.
The proof is completed using the Mayer--Vietoris sequence (see Proposition \ref{prop_MV squence_Local}) and the Five Lemma.
\end{proof}

According to Proposition \ref{thm_local_index_isom}, the following result implies Theorem \ref{main_result}.
\begin{thm}
The map
$$ev_*:\lim_{d \to \infty} K_*(C_L^*(P_d(G), G, B))\to \lim_{d \to \infty} K_*(C^*(P_d(G), G, B))$$
induced by the evaluation-at-zero map on $K$-theory is injective.
\end{thm}
\begin{proof}
We have the following commutative diagram
\begin{equation*}
\xymatrix{
	\displaystyle\lim_{d \to \infty}K_{\ast}(C^{\ast}_L(P_d(G), G, B))\ar[d]^{(\widetilde{c}_L)_*}\ar[r]^{ev_{\ast}}& \displaystyle\lim_{d \to \infty}K_{\ast}(C^*(P_d(G), G, B))\ar[d]^{\widetilde{c}_*}\\
	\displaystyle\lim_{d \to \infty}K_{\ast}(C^{\ast}_{L}(P_d(G), G, C(Y \times X)\widehat{\otimes}B))\ar[r]^{ev_{\ast}}& \displaystyle\lim_{d \to \infty}K_{\ast}(C^*(P_d(G), G, C(Y \times X)\widehat{\otimes}B)).
}
\end{equation*}
Since the map $(\widetilde{c}_L)_*$ is an isomorphism, and the lower horizontal map is injective, the upper horizontal map is injective.
\end{proof}

\subsection*{Acknowledgement} The author would like to acknowledge Guoliang Yu for his invaluable guidance and constant support. He also thanks Zhizhang Xie and Geng Tian for useful conversations and thanks the anonymous referees whose comments and suggestions helped to greatly improve the paper.



\bibliographystyle{abbrv}
\bibliography{ref}

\end{document}